\documentclass[12pt]{amsart}

\usepackage{hyperref}

\usepackage{amsmath,amstext,amssymb,amsopn,amsthm}
\usepackage{url,verbatim}
\usepackage{mathtools,MnSymbol}
\usepackage{enumerate,afterpage}
\usepackage{subfig,caption}

\usepackage{color,graphicx,scalerel}

\usepackage[margin=30mm]{geometry}
\usepackage{eucal,mathrsfs,dsfont}

\usepackage{tikz}
\usetikzlibrary{shapes}
\newcommand*\circled[1]{\tikz[baseline=(char.base)]{
            \node[shape=circle,draw,inner sep=2pt] (char) {#1};}}

\allowdisplaybreaks

\newtheorem{theorem}{Theorem}[section]

\newtheorem{lemma}[theorem]{Lemma}
\newtheorem{proposition}[theorem]{Proposition}

\theoremstyle{definition}

\newtheorem*{notation}{Notation}

\theoremstyle{remark}

\numberwithin{equation}{section}

\newcommand{\eps}{\varepsilon}

\newcommand{\calF}{\mathcal{F}}

\newcommand{\calG}{\mathcal{G}}

\newcommand{\calD}{\mathcal{D}}

\newcommand{\calT}{\mathcal{T}}

\newcommand{\calB}{\mathcal{B}}

\newcommand{\calH}{\mathcal{H}}

\newcommand{\bone}{\mathds{1}}
\newcommand{\E}{\operatorname{\mathds{E}}} 
\renewcommand{\P}{\operatorname{\mathds{P}}} 

\newcommand{\bS}{\mathbf{S}}
\newcommand{\tnw}{{\calT_{{\scaleobj{0.7}{\rotatebox[origin=c]{-45}{$\updownarrow$}}}}}}
\newcommand{\tnwc}{{\calT^c_{{\scaleobj{0.7}{\rotatebox[origin=c]{-45}{$\updownarrow$}}}}}}

\title[Contradictory predictions]{Contradictory predictions}
\author{Krzysztof Burdzy and Soumik Pal}

\address{ Department of Mathematics, Box 354350, University of Washington, Seattle, WA 98195}
\email{burdzy@uw.edu}

\email{soumik@uw.edu}

\thanks{KB: Research supported in part by Simons Foundation Grant 506732. SP: Research supported in part by NSF Grant DMS-1612483.}

\keywords{Conditional probability, opinion, joint distribution of conditional expectations}

\subjclass[2010]{60E15}

\begin{document}

\begin{abstract}
We prove the sharp bound for the probability that two experts who have access to different information, represented by different $\sigma$-fields, will give radically different estimates of the probability of an event. This is relevant when one combines predictions from various experts in a common probability space to obtain an aggregated forecast. Our proof is constructive in the sense that, not only the sharp bounds are proved, but also the optimizer is constructed via an explicit algorithm. 
\end{abstract}

\maketitle

\section{Introduction}\label{intro}


Imagine two experts, with access to different information, but sharing the same worldview. We model this by a probability space $(\Omega, {\calF}, \P)$ and with two distinct sub-$\sigma$-fields $\calG$ and $\calH$ of $\calF$. The sub-$\sigma$-fields represent the information accessible to the two experts while the common probability space represents their worldview, in the sense that, if one of the experts knew exactly what the other knows, he/she would arrive at exactly the same conclusions. This set-up is sometimes called the problem of combining experts' opinions under partial information, or more colloquially ``wisdom of the crowds'', and is a popular topic in statistics and decision theory. In general there are $N$ experts with their individual sub-$\sigma$-fields who are all trying to predict the probability of a common event $A$ (such as a particular candidate will win the election). The usual question is if there is a \textit{coherent} way to combine their predictions to come up with a better forecast. Introduced this way in the mid 80's onwards, see \cite{genest1984, DeGroot88, Dawid1995}, such combinations typically take the form of weighted averages (\cite{degroot91}). The field has found a renewed interest in the current age of social networks (see \cite{Lorenz9020, French2011}). In particular, \cite{gneiting10, gneiting2013} recommend both linear and non-linear combinations, \cite{Pemantle16} develops a mathematical framework to combine predictions when experts use ``partially overlapping information sources'', and \cite{Dawid18} uses it for the case of $N=2$ experts in prediction markets who take turn in updating their beliefs. Also see \cite{MB15, casarin16, KAPETANIOS2015150} for applications to economics, \cite{KRUGER2016S172} for applications to banking and finance, \cite{MG16} for applications to meteorology, \cite{TAYLOR2018877} for applications to maintenance of wind turbines, and \cite{KB_S, Easwaran2016} for philosophical implications. The problem is also related to modeling insider trading in finance \cite{KP15} where the insider has more information that the rest of the traders, i.e., $\calG \subseteq \calH$, although the general non-containment scenario makes sense for two different insiders.

We ask a complementary but different question, not about the weighted average of the various predictions, but their spread. More specifically, we consider $N=2$ experts predicting the possibility of a common event $A$ and derive sharp probabilistic bounds on the range of their predictions. In particular, we are interested when this spread is large, a phenomenon that we call ``contradictory predictions''. Two people make contradictory predictions if one of them asserts that the probability of $A$ is very small and the other one says that it is very large. We will present a theorem formalizing the idea that the two experts are unlikely to make contradictory
predictions even if they have different information sources.

Let us formulate the problem in the quantitative and rigorous manner.

Let $A$ be an event in some probability space $(\Omega
,{\calF}, \P)$, so $A \in\calF$, and let $X = \P(A\mid {\calG})$
and $Y = \P(A\mid {\calH})$ for two sub-$\sigma$-fields $\calG$ and $\calH$ of $\calF$.  
Let
\begin{align}\label{s28.2}
\lambda(\delta) = \sup \P(|X-Y| \geq 1-\delta) ,
\end{align}
where the supremum is taken over all probability spaces  $(\Omega,{\calF}, \P)$,
all events $A\in \calF$ and all sub-$\sigma$-fields $\calG$ and $\calH$ of $\calF$.  
It was proved in Theorem 14.1 of \cite{KB_S} that $\lambda(\delta)\leq 5 \delta$ for $\delta \leq 1/10$. 
The following stronger result was proved by Jim Pitman and published in \cite[Thm. 18.1]{KB_R}, with his permission.
For all $\delta \in (0,1)$,
\begin{align}\label{f11.8}
\frac{2 \delta} {1+\delta} \leq \lambda(\delta)\leq 2 \delta.
\end{align}
The purpose of this article is to remove the gap between the lower and upper bounds, even though the gap is very small for small $\delta$, i.e., in the most interesting case.

\begin{theorem}\label{f9.1}
For all $\delta \in (0,1/2)$,
\begin{align}\label{s7.2}
\lambda(\delta) = \frac{2 \delta} {1+\delta} .
\end{align}
\end{theorem}

\medskip


It has been shown in \cite{BurPit} that $\lambda(\delta)$ is, curiously, discontinuous at $\delta=1/2$. More precisely, $\lambda(\delta) =1$ for $\delta\geq 1/2$. To see this, construct $X$ and $Y$ so that $\P(X=1/2) = 1$ and $\P(Y=0) = \P(Y=1)=1/2$ (see the discussion around (1.10) in \cite{BurPit}).

As mentioned above, one can  interpret $X$ and $Y$ as the opinions of two experts about the probability of $A$ given different sources of information $\calG$ and $\calH$,
assuming the experts agree on some initial assignment of probability $\P$ to events in $\calF$.
The authors of \cite[p. 284]{Dawid1995} pointed out that
\begin{quote}
If $X$ and $Y$ are both produced by ``experts'', then one should not expect them to be wildly different. For example, it would seem paradoxical if, with $X$ say uniform on $[0,1]$,
one always had $Y = 1 - X$. 
\end{quote}

This suggests that $X$ and $Y$ cannot be too negatively dependent. However, 
elementary examples in \cite[\S 4.1]{Dawid1995}
show that for any prescribed value of $EX = EY = P(A) \in (0,1)$, the correlation 
between  $X$ and $Y$  can take
any value in $(-1,1]$. 
Consider for instance, for $\delta \in (0,1)$, the distribution of $(X,Y)$
concentrated on the three points $(1-\delta, 1 - \delta)$, $(0, 1-\delta)$ and $(1-\delta,0)$, with 
\begin{equation*}
P(X=Y) = P(1-\delta,1-\delta) = \frac{ 1 - \delta}{ 1 + \delta } \qquad \mbox{and} \qquad P(0, 1-\delta)= P(1-\delta,0) = \frac{ \delta}{ 1 + \delta }  .
\end{equation*}
This example from \cite{MR553386} 
gives   $X$ and $Y$ with
correlation $\rho(X,Y) = - \delta$ which can be any value in $(-1,0)$.

\medskip

We end this section with  a proof of \eqref{f11.8} borrowed from \cite[Thm. 18.1]{KB_R} because it is simple and it underscores the huge gap between the complexity of the proof of the upper bound in \eqref{f11.8} and our proof of the upper bound in  \eqref{s7.2}.

\begin{proposition} \label{s28.4}
(J.~Pitman)
For all $\delta \in (0,1)$,
\begin{align}\label{s28.3}
\frac{2 \delta} {1+\delta} \leq \lambda(\delta)\leq 2 \delta.
\end{align}
\end{proposition}

\begin{proof}
We will use notation matching that in the proof of our main theorem.

The lower bound for $\lambda(\delta)$ is attained by the following example. Fix $0 < \delta < 1$. Let
$\{C_{1,1}, C_{1,2}, C_{2,1}\}$ be a partition of $\Omega$, with 
$\P(C_{1,2})=\P(C_{2,1})=\delta/(1 + \delta)$.
Let $A = C_{1,2}\cup C_{2,1}$.
Let $G_1 = C_{1,1}\cup C_{1,2}$, $G_2 = C_{2,1}$, $H_1 = C_{1,1}\cup C_{2,1}$, and $H_2 = C_{1,2}$. 
Let $\calG$ be
generated by the partition $\{G_1,G_2\}$ and let $\calH$ be generated by the partition $\{H_1,H_2\}$.  It is clear by construction
that
\begin{align*}
X= \P(A\mid {\calG}) = \delta \bone_{G_1} + \bone_{G_2}, \qquad Y= \P(A\mid {\calH}) = \delta \bone_{H_1} + \bone_{H_2}.
\end{align*}
It follows that
\begin{align*}
|X-Y| = (1-\delta) \bone_{C_{1,2}\cup C_{2,1}} = (1-\delta) \bone_A
\end{align*}
and hence
\begin{align*}
\P(|X-Y| \geq 1-\delta) = \P(A) = 2\delta/(1 + \delta).
\end{align*}

\bigskip
 Next we will prove the upper bound. Assume that $0 < \delta < 1/2$. Note that for any $0 < \delta < 1/2$ and any random variables $X$ and $Y$ with $0\leq X \leq 1$ and $0\leq Y \leq 1$,
\begin{align}\label{s7.3}
\{|X-Y| \geq 1-\delta\} \subset \{ X\leq \delta, Y \geq 1-\delta\}
\cup \{ Y\leq \delta, X \geq 1-\delta\}.
\end{align}
Since $X = \E(\bone_A \mid X)$,
\begin{align*}
\P(X\leq \delta, Y \geq 1-\delta, A) 
\leq \P(X\leq \delta,  A) 
 = \E(\bone_{\{X\leq \delta\}}X)
\leq \delta \P(X\leq \delta).
\end{align*}
We have $1 - Y = \E(\bone_{A^c} \mid Y )$ so
\begin{align*}
\P(X\leq \delta, Y \geq 1-\delta, A^c) 
&\leq \P(Y \geq 1-\delta, A^c) 
 = \E(\bone_{\{1-Y\leq \delta\}}(1-Y)) \\
&\leq \delta \P(Y\geq 1-\delta).
\end{align*}
It follows that
\begin{align}\label{s7.4}
\P(X\leq \delta, Y \geq 1-\delta)
\leq \delta ( \P(X\leq \delta)
+ \P(Y\geq 1- \delta))
\end{align}
and similarly
\begin{align}\label{s7.5}
\P(Y\leq \delta, X \geq 1-\delta)
\leq \delta ( \P(Y\leq \delta)
+ \P(X\geq 1- \delta)).
\end{align}
For $0 < \delta < 1/2$ the events $X \leq \delta$ and $X \geq 1- \delta$ are disjoint, so $\P(X\leq \delta) + \P(X\geq 1-\delta) \leq 1$,
and the same holds for $Y$. Add (\ref{s7.4}) and (\ref{s7.5}) and use (\ref{s7.3}) to obtain the upper bound in (\ref{s28.3}).
\end{proof}

\section{Proof of  Theorem \ref{f9.1}}

The proof of Theorem \ref{f9.1}, our main result, will consist of a sequence  of lemmas.

Fix any $\delta \in (0, 1/2)$. Most elements of the model (sets, probabilities) will change within this section but the value of  $\delta$ will remain fixed.

\begin{lemma}\label{f11.7}
Consider any two events $G$ and $H$ such that $\P(G)>0$ and $\P(H)>0$.
If 
\begin{align*}
|\P(A\mid G) - \P(A\mid H)| \geq 1-\delta
\end{align*}
then
\begin{align}\label{j30.2}
\P(G\cap H) \leq \frac \delta{1+\delta} (\P(G) + \P(H)).
\end{align}
\end{lemma}

\begin{proof}
Assume without loss of generality that 
\begin{align*}
\P(A\mid G) - \P(A\mid H) \geq 1-\delta.
\end{align*}
Suppose that we can prove that
\begin{align}\label{f18.1}
\P( (G\cap H^c) \cup (H\cap G^c)\mid G\cup H) \geq \P(A\mid G) - \P(A\mid H).
\end{align}
Then
\begin{align*}
\frac{\P( G\cap H)}{\P(  G\cup H)}&=
\P( G\cap H\mid G\cup H) = 1 - \P( (G\cap H^c) \cup (H\cap G^c)\mid G\cup H)\\
&\leq 1- (\P(A\mid G) - \P(A\mid H)) \leq \delta,
\end{align*}
and, therefore,
\begin{align*}
&\P( G\cap H) \leq \delta \P(  G\cup H)
= \delta (\P(G) + \P(H)) - \delta \P( G\cap H),\\
&(1+\delta)\P( G\cap H) \leq \delta (\P(G) + \P(H)).
\end{align*}
This implies \eqref{j30.2} so it will suffice to prove \eqref{f18.1}.

Let 
\begin{alignat*}{3}
&p_1 = \P(A \mid G\cap H^c), \qquad
&&p_2 = \P(A \mid G\cap H),  \qquad
&&p_3 = \P(A \mid G^c\cap H), \\
&q_1 = \P( G\cap H^c),  \qquad
&&q_2 = \P( G\cap H),  \qquad
&&q_3 = \P( G^c\cap H). 
\end{alignat*}
Then
\begin{align}\label{f18.2}
\P(A\mid G) - \P(A\mid H)
= \frac{p_1 q_1 + p_2q_2}{q_1+q_2}
- \frac{p_2 q_2 + p_3q_3}{q_2+q_3}
\leq
\frac{ q_1 + p_2q_2}{q_1+q_2}
- \frac{p_2 q_2 }{q_2+q_3}.
\end{align}
For $p_2=0$ we obtain
\begin{align}\label{f18.3}
\frac{ q_1 + p_2q_2}{q_1+q_2}
- \frac{p_2 q_2 }{q_2+q_3}
=
\frac{ q_1 }{q_1+q_2} \leq \frac{q_1+q_3}{q_1+q_2+q_3}
= \P( (G\cap H^c) \cup (H\cap G^c)\mid G\cup H).
\end{align}
For $p_2=1$,
\begin{align}\label{f18.4}
\frac{ q_1 + p_2q_2}{q_1+q_2}
- \frac{p_2 q_2 }{q_2+q_3}
=
\frac{ q_1+q_2 }{q_1+q_2}- \frac{ q_2 }{q_2+q_3} \leq \frac{q_1+q_3}{q_1+q_2+q_3}
= \P( (G\cap H^c) \cup (H\cap G^c)\mid G\cup H).
\end{align}
The right hand side of \eqref{f18.2} is a linear function of $p_2$ so
\eqref{f18.2}-\eqref{f18.4} imply that for all $p_2\in[0,1]$,
\begin{align*}
\P(A\mid G) - \P(A\mid H)
\leq
\frac{ q_1 + p_2q_2}{q_1+q_2}
- \frac{p_2 q_2 }{q_2+q_3}
\leq \P( (G\cap H^c) \cup (H\cap G^c)\mid G\cup H).
\end{align*}
This completes the proof of \eqref{f18.1}.
\end{proof}

\begin{lemma}\label{s28.5}
Recall $\lambda(\delta)$ defined in \eqref{s28.2} and
let
\begin{align*}
\lambda'(\delta) = \sup \P(|X-Y| \geq 1-\delta) ,
\end{align*}
where the supremum is taken over all probability spaces  $(\Omega,{\calF}, \P)$,
all events $A\in \calF$ and all \emph{finitely} generated sub-$\sigma$-fields $\calG$ and $\calH$ of $\calF$. If $\lambda'(\delta) =  2 \delta/(1+\delta)$ for all $\delta\in(0,1/2)$ then $\lambda(\delta) =  2 \delta/(1+\delta)$ for all $\delta\in(0,1/2)$.
\end{lemma}

\begin{proof}
It is obvious that $\lambda'(\delta) \leq \lambda(\delta)$.
It remains to prove the opposite inequality.

  For $n>1$, let $\calG_n$ be the $\sigma$-field generated by the events $ \{k/n \leq X <( k+1)/n\}$, $0 \leq k \leq n$ (note that $k=n$ is included). Let $X_n=\P(A\mid \calG_n)$. Since 
\begin{align*}
\P(\{X_n \notin [k/n, (k+1)/n)\} \cap \{k/n \leq X <( k+1)/n\}) = 0,
\end{align*}
we have
\begin{align*}
\P(|X_n - X| \leq 1/n) =1.
\end{align*}
Similarly, let $\calH_n$ be the $\sigma$-field generated by the events $ \{k/n \leq Y <( k+1)/n\}$, $0 \leq k \leq n$, and $Y_n=\P(A\mid \calH_n)$. Then
\begin{align*}
\P(|Y_n - Y| \leq 1/n) =1.
\end{align*}
It follows that 
\begin{align}\label{j20.1}
\P(|X - Y| \geq 1-\delta ) \leq
\P(|X_n - Y_n| \geq 1-\delta-2/n ).
\end{align}
Since $n$ can be arbitrarily large,
\eqref{j20.1} implies that if one can prove that 
$\lambda'(\delta) = 2 \delta/(1+\delta)$
for all $\delta\in(0,1/2)$ then $\lambda(\delta) =  2 \delta/(1+\delta)$ for all $\delta\in(0,1/2)$. 
\end{proof}
 
\begin{notation}
In view of Lemma \ref{s28.5}
 we can and will assume from now on that  $\calG$ and $\calH$ are finitely generated.
Let $\{G_1, G_2, \dots ,G_{m(\calG)}\}$ be the partition of $\Omega$ generating $\calG$ and let the family $\{H_1, H_2, \dots ,H_{m(\calH)}\}$ be defined in the analogous way relative to $\calH$.

We will assume without loss of generality that $\P(G_k)>0$ and $\P(H_k)>0$ for all $k$.

Let 
\begin{align}\label{f3.2}
p_k &= \P(G_k), \qquad 1\leq k \leq m(\calG),\\
q_k &= \P(H_k), \qquad 1\leq k \leq m(\calH),\label{f3.3}\\
C_{j,k} & = G_j \cap H_k, \\
x_k &= X(\omega), \qquad \text{  for  }\omega \in G_k,\  1\leq k \leq m(\calG),\label{f3.5}\\
y_k &= Y(\omega), \qquad \text{  for  }\omega \in H_k,\  1\leq k \leq m(\calH).\label{f3.6}
\end{align}
Events $C_{k,j}$ will be called cells.
Assume without loss of generality that
\begin{align}\label{f3.7}
x_1\leq x_2\leq \dots \leq x_{m(\calG)},\qquad
 y_1\leq y_2\leq \dots\leq y_{m(\calH)}.
\end{align}
The values of $x_k$ grow from the left to the right and values of $y_k$ grow from the bottom to the top in Fig.~\ref{fig10}.

Let 
\begin{align}\label{f7.4}
m_-(\calG) & = \max \{k: \exists j \text{  such that  } y_j- x_k \geq 1-\delta\},\\
m_+(\calG) & = \min \{k:  \exists j \text{  such that  } x_k-y_j \geq 1-\delta\},\label{f7.5}\\
m_-(\calH) & = \max \{k:  \exists j \text{  such that  }x_j- y_k \geq 1- \delta\},\label{f7.6}\\
m_+(\calH) & = \min \{k:  \exists j \text{  such that  } y_k-x_j \geq 1-\delta\}.\label{f7.7}
\end{align}
By convention, $\max(\emptyset) = 0$ and $\min(\emptyset) = \infty$.
Let 
\begin{align}\label{f7.1}
B = \{|X-Y| \geq 1-\delta  \}.
\end{align}
The  definitions \eqref{f7.4}-\eqref{f7.7} are illustrated in Fig.~\ref{fig10} as follows. The ``projection'' of the set $B$ (the cells within two closed thick polygonal lines) on the horizontal axis consists of two disjoint intervals $[1, m_-(\calG)]$ and $[m_+(\calG), m(\calG)]$. The ``projection'' of the set $B$  on the vertical axis consists of two disjoint intervals $[1, m_-(\calH)]$ and $[m_+(\calH), m(\calH)]$.

\begin{figure} \includegraphics[width=1.0\linewidth]{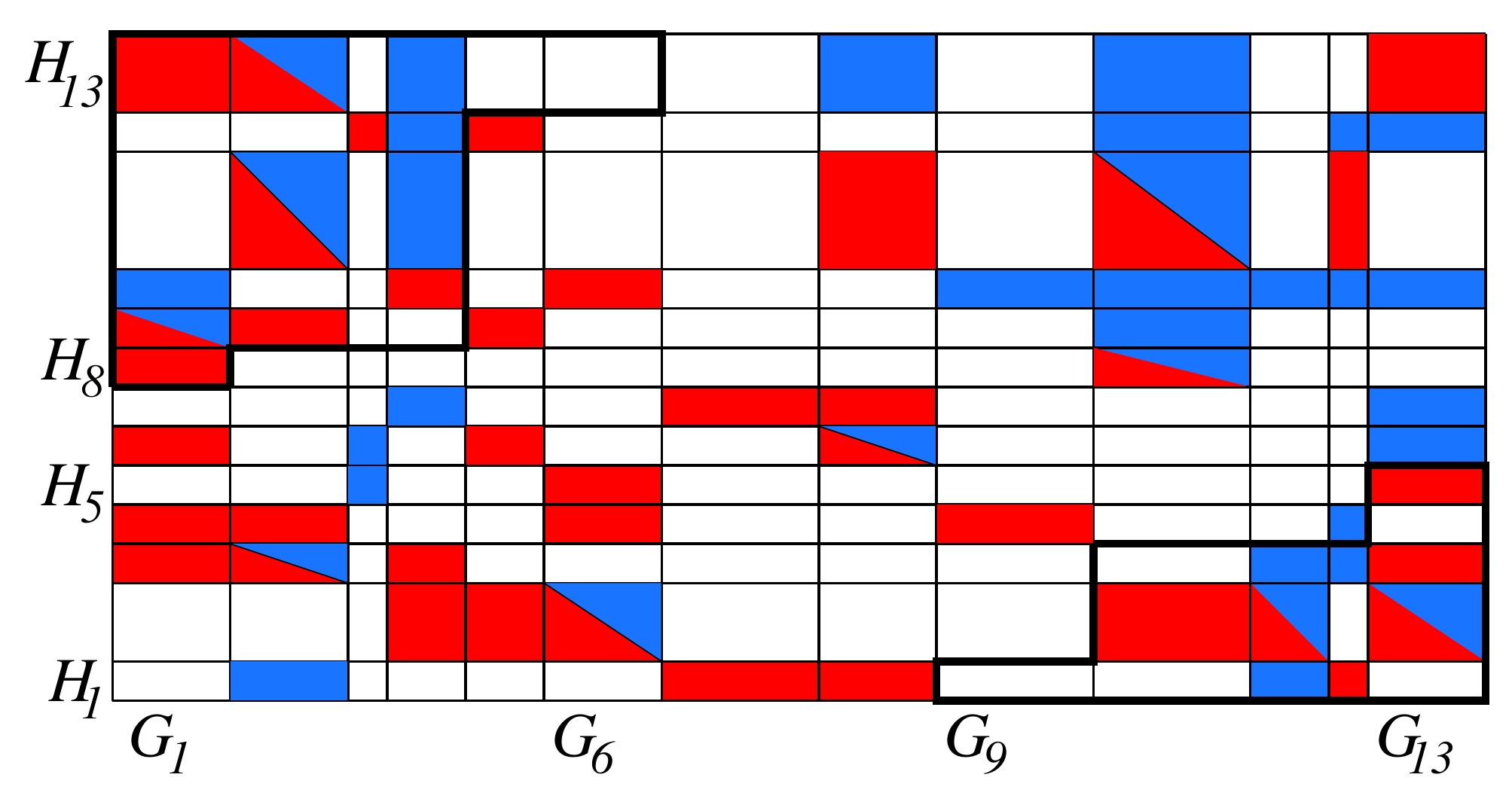}
\caption{(Color coded)
White color represents ``empty'' cells, i.e., cells $C_{k,j}$ such that $\P(C_{k,j})=0$. Red color designates ``cells in $A^c$,'' i.e., cells $C_{k,j}$ such that $\P(C_{k,j} )>\P(C_{k,j} \cap A) =0$. Blue color designates ``cells in $A$,'' i.e., cells $C_{k,j}$ such that $\P(C_{k,j} )>\P(C_{k,j} \cap A^c) =0$. Rectangles colored red and blue represent cells $C_{k,j}$ that contain both  $A$ and $A^c$, i.e., $\P(C_{k,j} \cap A) >0$ and $\P(C_{k,j} \cap A^c) >0$. 
 The event $B$ is 
contained within two thick closed polygonal lines.
In this illustration, $m(\calG) = m(\calH) = 13$, $m_-(\calG) = 6$, $m_+(\calG) = 9$, $m_-(\calH) = 5$, $m_+(\calH) = 8$. Cells have different areas to illustrate the point that their probabilities may be unequal but  no conditions on probability values should be inferred from cell areas. 
}
\label{fig10}
\end{figure}

We will define a number of transformations of the probability space $(\Omega, \calF, \P)$, event $A$, and $\sigma$-fields $\calG$ and $\calH$.
We will denote the transformed objects $(\Omega', \calF', \P)$, $A'$, $\calG'$ and $\calH'$. We will also write in a similar manner $X'$, $p_k'$, $x_k'$, etc. Note that the only exception to this rule is $\P$. We will write $\P$ instead of $\P'$ (except in Lemma \ref{o1.2}) even though the transformed probability measure is not necessarily equal to the original one. This should not cause any confusion.
We will denote our transformations $\calT_1$, $\calT_2$, etc., and we will write
\begin{align*}
\bS &= (\Omega, \calF, \P,A,\calG,\calH),\qquad
\bS' = (\Omega', \calF', \P,A',\calG',\calH'),\qquad
\calT(\bS) = \bS'.
\end{align*}
All objects that are not listed in the definition of $\bS$, for example $X$, $Y$ and those in \eqref{f3.2}-\eqref{f3.6}, are uniquely defined given $\bS$.

Our definitions of  transformations will contain some assumptions about $\bS$. If the assumptions are not satisfied, the transformation should be interpreted as the identity, i.e., $\calT(\bS) = \bS$. 

Note that
$\P(B) =
\sum_{1\leq j \leq m(\calG), 1\leq k \leq m(\calH)}
\P(C_{j,k}\cap B) $.
Most of our transformations will satisfy the following three conditions.
\begin{align}\label{f3.1}
\P(B')=
\sum_{1\leq j \leq m'(\calG'), 1\leq k \leq m'(\calH')}
\P(C'_{j,k}\cap B') &\geq
\sum_{1\leq j \leq m(\calG), 1\leq k \leq m(\calH)}
\P(C_{j,k}\cap B)  =\P(B).\\
m'(\calG')\leq m(\calG), \qquad &m'(\calH')\leq m(\calH).
\label{s28.6}
\end{align}
If
\begin{align}\label{o1.1}
\sum_{1\leq j \leq m_-(\calG), 1\leq k \leq m(\calH)} \P(C_{j,k}\cap B)>0
\ \text{and} \ 
\sum_{m_+(\calG) \leq j \leq m(\calG), 1\leq k \leq m(\calH)} \P(C_{j,k}\cap B)>0,
\end{align}
then
\begin{align} \label{o2.1} 
&\sum_{1\leq j \leq m'_-(\calG'), 1\leq k \leq m'(\calH')} \P(C'_{j,k}\cap B')>0
\ \text{and} \ 
\sum_{m'_+(\calG') \leq j \leq m'(\calG'), 1\leq k \leq m'(\calH')} \P(C'_{j,k}\cap B')>0. 
\end{align}

The meaning of condition \eqref{o1.1}-\eqref{o2.1} is that if the part of $B$ in the left upper corner in Fig.\ref{fig10} is non-empty and has a strictly positive $\P$-measure and the same is true for the part of $B$ in the lower right corner, then the same holds for  $\bS'$.

\end{notation}

\begin{lemma}\label{o1.2}
Suppose that $\bS$ is given and $\P(B) >0$. Then for every $\eps>0$ there exists $\bS'$ such that $\P(B')> \P(B) - \eps$, $m'(\calG') \leq m(\calG)+1$,  $m'(\calH') \leq m(\calH)+1$,  and \eqref{o2.1} holds.
\end{lemma}

\begin{proof}
Since $\P(B)>0$, we must have 
$ m_-(\calG) \geq 1$, $ m_+(\calH) \leq m(\calH)$ and
\begin{align}\label{o3.1}
\sum_{1\leq j \leq m_-(\calG), 1\leq k \leq m(\calH)} \P(C_{j,k}\cap B)>0,
\end{align}
or
$ m_-(\calH) \geq 1$,
$m_+(\calG) \leq m(\calG)$ and
\begin{align}\label{o3.2}
\sum_{m_+(\calG) \leq j \leq m(\calG), 1\leq k \leq m(\calH)} \P(C_{j,k}\cap B)>0,
\end{align}
By symmetry, it will suffice to discuss only one of these cases. 

Suppose that $ m_-(\calG) \geq 1$, $ m_+(\calH) \leq m(\calH)$ and \eqref{o3.1} holds.
If $ m_-(\calH) \geq 1$,
$m_+(\calG) \leq m(\calG)$ and \eqref{o3.2} is true then we set $\bS'=\bS$ and we are done.
Suppose otherwise. We will assume that $ m_-(\calH) =0$ and
$m_+(\calG) =\infty$. It is easy to see that the argument given below applies also in the case when $ m_-(\calH) \geq 1$,
$m_+(\calG) \leq m(\calG)$ and \eqref{o3.2} does not hold.

Let
\begin{align*}
\Omega' = \Omega \cup \{\omega'_{k,0}: 1\leq k \leq m(\calG) +1\}
\cup \{\omega'_{m(\calG)+1,j}: 1\leq j \leq m(\calH) \},
\end{align*}
where all $\omega'_{k,j}$'s are distinct and they do not belong to $\Omega$.
Let 
\begin{align*}
C'_{k,j} &= C_{k,j}, \qquad  1\leq k \leq m(\calG), 1\leq j \leq m(\calH),\\
C'_{k,0} &= \{\omega'_{k,0} \}, \qquad 1\leq k \leq m(\calG) +1,\\
C'_{m(\calG)+1,j}& = \{\omega'_{m(\calG)+1,j}\}, \qquad 1\leq j \leq m(\calH) ,\\
G'_k &= \bigcup_{ 0\leq j \leq m(\calH)} C'_{k,j}, \qquad 1\leq k \leq m(\calG) +1,\\
H'_j &= \bigcup_{ 1\leq k \leq m(\calG)+1} C'_{k,j}, \qquad 0\leq j \leq m(\calH) ,\\
\calG' &= \sigma\{G'_k,  1\leq k \leq m(\calG) +1\},\\
\calH'&= \sigma\{H'_j, 0\leq j \leq m(\calH)\},\\
A' &= A \cup G'_{m(\calG)+1} .
\end{align*}

Fix any $\eps>0$.
Let $\eps_1>0$ be so small that
\begin{align}\label{o3.3}
\P(B) (1- \eps_1/2 - \eps_1 /4)
> \P(B) -\eps.
\end{align}
It is elementary to see that we can define a probability measure $\P'$ on $\Omega'$ with the following properties.
\begin{align*}
\P'\left(H'_0 \setminus C'_{m(\calG)+1, 0}\right) &= \eps_1/2,\\
\P'\left( C'_{m(\calG)+1, 0}\right)& >0,\\
\P'\left( G'_{m(\calG)+1} \right) & = \eps_1\delta/4,\\
\P'(F) &=  (1- \eps_1/2 - \eps_1 \delta /4) \P(F), \qquad F\subset \Omega.
\end{align*}
Note that $x_{m(\calG)+1} = 1$ and 
\begin{align*}
y_0 = \frac{\P(A'\cap H'_0)}{\P( H'_0)}
\leq \frac{\P\left(C'_{m(\calG)+1, 0}\right)}
{\P\left(H'_0 \setminus C'_{m(\calG)+1, 0}\right)}
\leq \frac{\P\left(G'_{m(\calG)+1}\right)}
{\P\left(H'_0 \setminus C'_{m(\calG)+1, 0}\right)}
= \frac{\eps_1 \delta /4}{\eps_1/2}= \delta/2.
\end{align*}
Hence, $C'_{m(\calG)+1, 0} \subset B'$. 
Let us shift the index for the generators $H'_j$ of $\calH'$ so that the generators are  $H'_1, H'_2, \dots, H'_{m(\calH)+1}$.
We see that $m'(\calG') = m(\calG)+1$,  $m'(\calH') = m(\calH)+1$, $ m'_-(\calH') \geq 1$,
$m'_+(\calG') \leq m'(\calG')$ and \eqref{o2.1} holds.
It remains to note that, in view of \eqref{o3.3},
\begin{align*}
\P'(B') \geq \P'(B) = \P(B) (1- \eps_1/2 - \eps_1 \delta /4)
> \P(B) -\eps.
\end{align*}
\end{proof}

\begin{lemma}\label{s29.5}
Suppose that for some $\bS$, $j_1$ and $k_1$, $\P(C_{j_1,k_1}\cap A) = p_1>0$, and consider any $p_2 \in (0, p_1)$.
There exists $\bS'$ such that $m'(\calG')=m(\calG)$, $m'(\calH')=m(\calH)$,
\begin{align*}
\P(C'_{j,k}\cap A') &= \P(C_{j,k}\cap A), \qquad
\P(C'_{j,k}\cap (A')^c) = \P(C_{j,k}\cap A^c),
\end{align*}
for all $1\leq j \leq m(\calG) $ and $ 1\leq k \leq m(\calH)$, and
there exists an event $F\subset C'_{j_1,k_1}\cap A'$ with $\P(F) = p_2$.
Moreover, conditions \eqref{f3.1}-\eqref{o2.1} are satisfied.
\end{lemma}

\begin{proof}
Let $\Omega'$ be defined as follows, 
\begin{align*}
\Omega'
= \{\omega'_{j,k,1}\}_{1\leq j \leq m(\calG), 1\leq k \leq m(\calH)}
\cup
\{\omega'_{j,k,2}\}_{1\leq j \leq m(\calG), 1\leq k \leq m(\calH)}
 \cup
\{\omega'_{j_1,k_1,3}\},
\end{align*}
where all listed elements are distinct. Let
\begin{align*}
C'_{j,k}\cap A' &= \{\omega'_{j,k,1}\},
\qquad C'_{j,k}\cap (A')^c = \{\omega'_{j,k,2}\},
\end{align*}
for all $1\leq j \leq m(\calG) $ and $ 1\leq k \leq m(\calH)$,
except that $C'_{j_1,k_2}\cap A' =\{\omega'_{j_1,k_1,1}, \omega'_{j_1,k_1,3}\}$. Let $F = \{\omega'_{j_1,k_1,1}\}$.

Define the probability measure on $\Omega'$ by
\begin{align*}
\P(C'_{j,k}\cap A') &= \P(C_{j,k}\cap A), \qquad
\P(C'_{j,k}\cap (A')^c) = \P(C_{j,k}\cap A^c),
\end{align*}
for all $1\leq j \leq m(\calG) $ and $ 1\leq k \leq m(\calH)$,
and
\begin{align*}
\P(F) = \P(\{\omega'_{j_1,k_1,1}\})=p_2,\qquad
\P(\{\omega'_{j_1,k_1,3}\})=p_1-p_2.
\end{align*}
It is elementary to check that $\bS'$ satisfies the lemma.
\end{proof}

\begin{lemma}\label{s29.1}
 Suppose that there exists $k\in\{1,\dots, m(\calG)-1\}$  such that for every $j\in\{1,\dots, m(\calH)\}$
we have 
\begin{align}\label{a11.5}
 C_{k,j}\cup C_{k+1,j} \subset B 
 \quad \text{  or  }\quad
 \P(( C_{k,j}\cup C_{k+1,j}) \cap B )=0.
\end{align}
Then there exists $\bS'$ such that \eqref{f3.1} and \eqref{o1.1}-\eqref{o2.1} are satisfied, and
\begin{align}
m'(\calG')= m(\calG)-1, \qquad &m'(\calH')\leq m(\calH).
\label{o4.1}
\end{align}
\end{lemma}

In graphical terms, the condition \eqref{a11.5} means that the thick polygonal line extends in a straight fashion along the boundaries of at least two consecutive cells in the interior of the rectangle in Fig.~\ref{fig10}.
There are several such ``long'' thick line segments in Fig.~\ref{fig10}.

\begin{proof}[Proof of Lemma \ref{s29.1}]
Let
\begin{align}
G'_j & = G_j, \quad 1\leq j < k,\label{a11.2}\\
G'_k & = G_k \cup G_{k+1}, \label{a11.6}\\
G'_j & = G_{j+1}, \quad k+1 \leq j \leq m(\calG)-1.\label{a11.3}
\end{align}
Let $\calG'$ be the $\sigma$-field generated by $\{G'_j\}_{1\leq j \leq m(\calG)-1}$.  All other objects in $\bS$ remain unchanged, for example, $A'=A$,  $\calH' = \calH$, etc. 
It follows from \eqref{a11.2}-\eqref{a11.3} that  for 
$1\leq j \leq m(\calH)$,
\begin{align*}
C'_{r,j} &= C_{r,j}, \quad 1\leq r < k,\notag\\
C'_{k,j} &= C_{k,j} \cup C_{k+1,j},\\
C'_{r,j} &= C_{r+1,j}, \quad k+1 \leq r \leq m(\calG)-1.\notag
\end{align*}
 Hence, to prove \eqref{f3.1}, it will suffice to show that for all $j$,
\begin{align}\label{a11.1}
\P(C'_{k,j}\cap B) 
\geq
\P(C_{k,j}\cap B) + \P(C_{k+1,j}\cap B).
\end{align}
It is elementary to check that \eqref{a11.6} implies that $x_k \leq x'_k \leq x_{k+1}$. In view of \eqref{a11.5}, for every $j$, we have either
\begin{align}\label{a11.7}
|x_k-y_j| \geq 1-\delta \quad\text{  and  }\quad
|x_{k+1}-y_j| \geq 1-\delta ,
\end{align}
or
\begin{align}\label{a11.8}
|x_k-y_j| < 1-\delta \quad\text{  and  }\quad
|x_{k+1}-y_j| < 1-\delta .
\end{align}
If \eqref{a11.7} is true then $|x'_k-y_j| \geq 1-\delta$, so $C'_{k,j}\subset B$ and, therefore, \eqref{a11.1} is satisfied.
In the case \eqref{a11.8}, the condition \eqref{a11.1} holds because the right hand side is 0.
We have finished the proof of \eqref{f3.1}. 

It follows from \eqref{a11.2}-\eqref{a11.3} that \eqref{o4.1} is satisfied. 

Some cells that belong to $B$ can coalesce in the upper left corner  in Fig.~\ref{fig10}, if there any to start with, but they cannot disappear. The same remark applies to the lower right corner so the condition \eqref{o1.1}-\eqref{o2.1} holds.
\end{proof}

\begin{lemma} \label{s29.2}
For any  $\bS$ there exists $\bS'$ such that \eqref{f3.1}-\eqref{o2.1} are satisfied and
\begin{align}
m'_-(\calG') &= m'(\calH')- m'_+(\calH')+1,\label{f6.8} \\
 m'_-(\calH') &= m'(\calG') -m'_+(\calG')+1,\label{f6.9}\\
y'_{m'(\calH') - k+1} - x'_k &\geq 1-\delta, \qquad 1\leq k \leq m'_-(\calG'), \label{f6.1} \\
y'_{m'(\calH') - k+1} - x'_{k+1} &< 1-\delta, \qquad 1\leq k \leq m'_-(\calG'), \label{f6.2} \\
y'_{m'(\calH') - k} - x'_k &< 1-\delta, \qquad 1\leq k \leq m'_-(\calG'), \label{f6.3} \\
x'_{m'(\calG') - k+1} - y'_k &\geq 1-\delta, \qquad 1\leq k \leq m'_-(\calH'), \label{f6.4} \\
x'_{m'(\calG') - k+1} - y'_{k+1} &< 1-\delta, \qquad 1\leq k \leq m'_-(\calH'), \label{f6.5} \\
x'_{m'(\calG') - k} - y'_k &< 1-\delta, \qquad 1\leq k \leq m'_-(\calH'), \label{f6.6} \\
\label{f3.12}
x'_1 < x'_2 < \dots < x'_{m(\calG')} & \quad \text{  and   }\quad
y'_1 < y'_2 < \dots < y'_{m(\calH')}.
\end{align}
Moreover,
\begin{align}
\bS' = \bS \qquad \text{or} \qquad 
m'(\calG')< m(\calG)\qquad \text{or} \qquad &m'(\calH')< m(\calH).
\label{o4.3}
\end{align}
\end{lemma}

\begin{proof}
We apply  the transformation defined in Lemma \ref{s29.1} repeatedly, as long as there is some $k$ satisfying \eqref{a11.5}. In view of \eqref{o4.1}, these transformations strictly  decrease $m(\calG)$ so they have to stop at some point, either when $m(\calG) =1$ or when \eqref{a11.5} is not satisfied any more. Then we exchange the roles of $\calG$ and $\calH$ and we collapse pairs of ``rows'' $H_k$ and $H_{k+1}$ into $H'_k$ in a similar manner until the condition analogous to \eqref{a11.5} fails.

If $\bS'$ is the final result of the transformations described above then $\bS'$ may be represented graphically as follows. 
The thick polygonal line must turn at every cell corner in the interior of the rectangle (see Fig.~\ref{fig1}), so that condition \eqref{a11.5} is not satisfied and neither is its analogue with the roles of $\calG$ and $\calH$ exchanged. In other words, the boundary of $B$ in the interior of the rectangle is a zigzag line that turns at every opportunity. It is elementary to check that
conditions \eqref{f6.8}-\eqref{f6.6} are a rigorous version of this graphical description. 

To see that \eqref{f3.12} holds, recall that we have weak inequalities in view of \eqref{f3.7}. If any weak inequality in  the first set  is an equality, say $x'_k= x'_{k+1}$, then 
\eqref{a11.5} is satisfied and we can reduce the number of generators $m'(\calG')$ using the transformation defined in Lemma \eqref{s29.1}. A similar argument applies to the second set of inequalities in \eqref{f3.12}. 
However, by the argument given in the first part of this proof, $m'(\calG')$ and $m'(\calH')$ cannot be decreased any more by the transformation defined in Lemma \eqref{s29.1}, so \eqref{f3.12} must be true.

Conditions \eqref{f3.1}-\eqref{o2.1} are satisfied because this is the case for the transformation defined in  Lemma \eqref{s29.1}. The alternative stated in \eqref{o4.3} follows from \eqref{o4.1}.
\end{proof}

\begin{notation}
If $(k,j) = (k, m(\calH) - k+1)$ and conditions \eqref{f6.1}-\eqref{f6.3}, without primes, are satisfied then we will write
$(k,j) \in\calD_-$.
If $(j,k) = ( m(\calG) - k+1,k)$ and conditions \eqref{f6.4}-\eqref{f6.6}, without primes, are satisfied then we will write
$(j,k) \in\calD_+$. In Fig.~\ref{fig1}, the cells belonging to $\calD_-$ and $\calD_+$ are crossed. The set $\calD_-\cup \calD_+$ is the ``internal border'' of $B$.
Note that if $(k,j) \in\calD_-$ then $x_k \leq \delta$ and $y_j \geq 1-\delta$. If $(k,j) \in\calD_+$ then $x_k \geq 1-\delta$ and $y_j \leq \delta$. 
\end{notation}

\begin{lemma}\label{s29.3}
Suppose that 
$\bS $ satisfies \eqref{f6.8}-\eqref{f3.12} and there exist $k\in\{1,\dots, m(\calG)-1\}$  and $i\in\{1,\dots, m(\calH)\}$ such that
\begin{align}\label{f5.1}
 &C_{k,i} \subset B , \quad C_{k+1,i} \subset B^c, \quad\P(C_{k,i})=0 .
\end{align}
Then there exists $\bS'$ such that \eqref{f3.1} and \eqref{o1.1}-\eqref{o2.1} are satisfied, and
\begin{align}
m'(\calG')= m(\calG)-1, \qquad &m'(\calH')\leq m(\calH).
\label{o4.2}
\end{align}
\end{lemma}

The condition \eqref{f5.1} is illustrated in Fig.~\ref{fig10} as follows: in row $H_9$, the   rightmost cell in $B$ (within thick polygonal line) is white, i.e., its probability is 0.
The transformation defined in Lemma \ref{s29.3} is illustrated in Figs. \ref{fig10} and \ref{fig11} (see the caption of Fig.~\ref{fig11}).
\begin{figure} \includegraphics[width=1.0\linewidth]{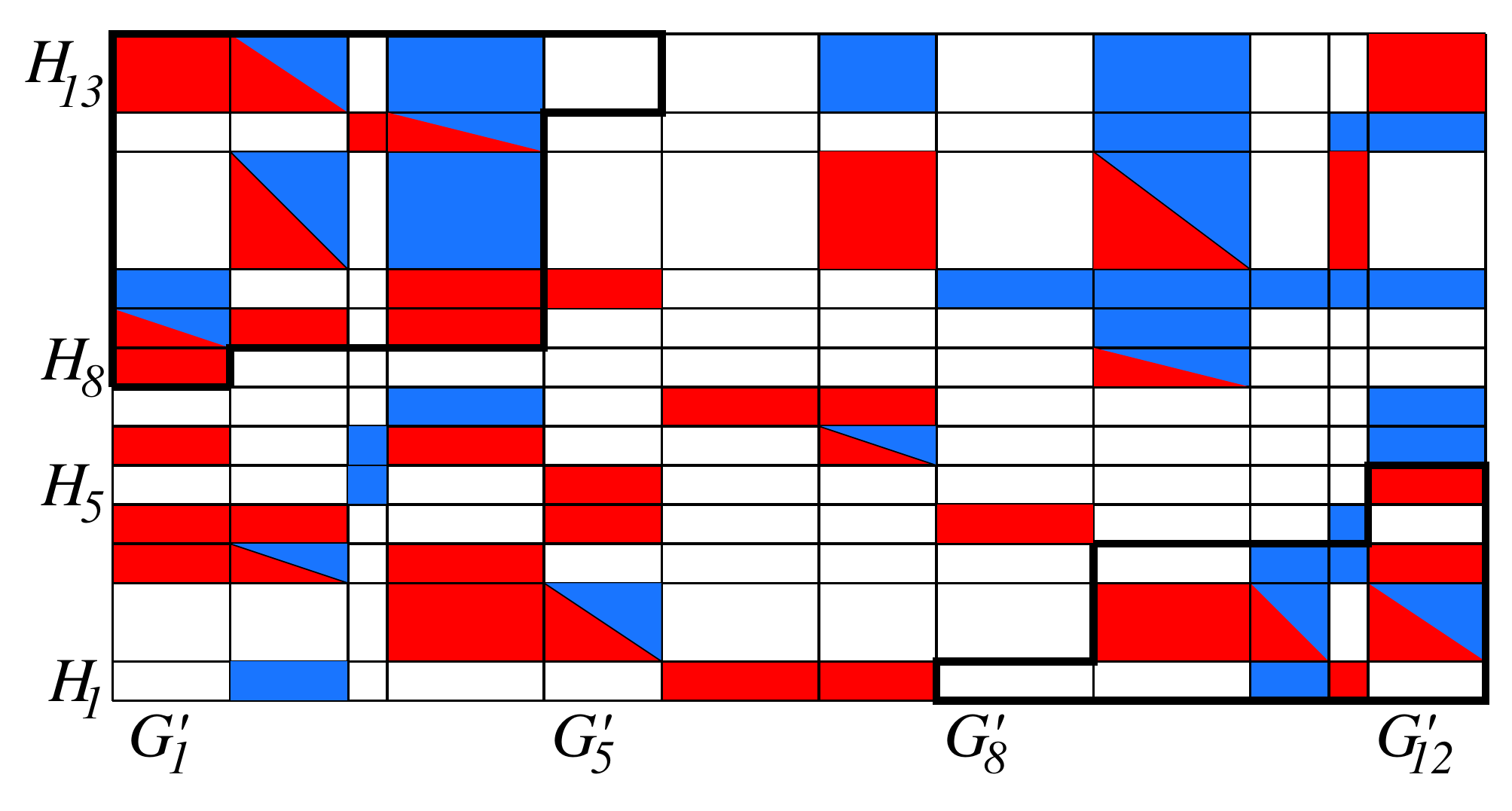}
\caption{ (Color coded)
This configuration $\bS'$ has been obtained by transforming $\bS$ in Fig.~\ref{fig10} according to \eqref{a11.21}.
For the meaning of colors and thick lines, see Fig.~\ref{fig10}.
Columns $ G_4$ and $G_5$ in Fig.~\ref{fig10} have been combined to form column $G'_4$ in the present figure. This is because in row $H_9$ in  Fig.~\ref{fig10}, the  rightmost cell in $B$ (within thick polygonal line) is white, i.e., its probability is 0.
}
\label{fig11}
\end{figure}

\begin{proof}[Proof of Lemma \ref{s29.3}]
Let
\begin{align}
G'_j & = G_j, \quad 1\leq j < k,\label{a11.20}\\
G'_k & = G_k \cup G_{k+1}, \label{a11.21}\\
G'_j & = G_{j+1}, \quad k+1 \leq j \leq m(\calG)-1.\label{a11.22}
\end{align}
Let $\calG'$ be the $\sigma$-field generated by $\{G'_j\}_{1\leq j \leq m(\calG)-1}$.  All other objects in $\bS$ remain unchanged, for example, $A'=A$,  $\calH' = \calH$, etc. 
It follows from \eqref{a11.20}-\eqref{a11.22} that  for 
$1\leq j \leq m(\calH)$,
\begin{align}\label{o3.4}
C'_{r,j} &= C_{r,j}, \quad 1\leq r < k,\\
C'_{k,j} &= C_{k,j} \cup C_{k+1,j},\label{o3.5}\\
C'_{r,j} &= C_{r+1,j}, \quad k+1 \leq r \leq m(\calG)-1.\label{o3.6}
\end{align}
 Hence, to prove \eqref{f3.1}, it will suffice to show that for all $j$,
\begin{align}\label{a11.23}
\P(C'_{k,j}\cap B') 
\geq
\P(C_{k,j}\cap B) + \P(C_{k+1,j}\cap B).
\end{align}
It is elementary to check that \eqref{a11.21} implies that $x_k \leq x'_k \leq x_{k+1}$.

Consider $j\ne i$. It follows from conditions \eqref{f6.8}-\eqref{f6.6} 
(see the remark about the zigzag line in the proof of Lemma \ref{s29.2}) that
\begin{align*}
 C_{k,j}\cup C_{k+1,j} \subset B 
 \quad \text{  or  }\quad
 \P(( C_{k,j}\cup C_{k+1,j}) \cap B )=0.
\end{align*}
In the first case,
\begin{align*}
|x_k-y_j| \geq 1-\delta \quad\text{  and  }\quad
|x_{k+1}-y_j| \geq 1-\delta .
\end{align*}
This implies that
$|x'_k-y'_j| =|x'_k-y_j|\geq 1-\delta$, so $C'_{k,j}\subset B'$ and, therefore, 
\begin{align*}
\P(C'_{k,j}\cap B') 
=\P(C'_{k,j}) 
= \P(C_{k,j}) + \P(C_{k+1,j})
\geq
\P(C_{k,j}\cap B) + \P(C_{k+1,j}\cap B).
\end{align*}
Hence
\eqref{a11.23} is satisfied.

If $\P(( C_{k,j}\cup C_{k+1,j}) \cap B )=0$ then \eqref{a11.23} holds because its right hand side  is 0. 

In the case when $j=i$, our assumptions \eqref{f5.1} imply that
$\P(C_{k,i}\cap B) = \P(C_{k+1,i}\cap B)=0$, so we conclude that \eqref{a11.23} is true. This completes the proof of \eqref{f3.1}.

It follows from \eqref{a11.20}-\eqref{a11.22} that  \eqref{o4.2} is satisfied.

Condition \eqref{o1.1}-\eqref{o2.1} follows from \eqref{o3.4}-\eqref{a11.23}.
\end{proof}

\begin{lemma} \label{o3.7}
For any  $\bS$ there exists $\bS'$ such that \eqref{f3.1}-\eqref{o2.1} and \eqref{f6.8}-\eqref{f3.12} are satisfied, and  $\P(C'_{i,j}) >0$ for all $(i,j) \in\calD'_- \cup\calD'_+$.
\end{lemma}

\begin{proof}
Let $\calT_1$ denote the transformation defined in Lemma \ref{s29.2}.

Lemma \ref{s29.3} can be, obviously, generalized to any situation when $C_{k,i} \subset B$, $\P(C_{k,i})=0$, and $C_{k,i}$ is adjacent to  a cell  in $B^c$. Fix any order, denoted $\prec$, for the set of pairs of cells. Let $\calT_2$ denote the transformation defined in Lemma \ref{s29.3}, applied to the first (according to the order $\prec$) pair of cells satisfying the generalized conditions described at the beginning of this paragraph.
Recall that, by convention, if there is no such pair of cells then $\calT_2(\bS) = \bS$.

Let $\bS_0 = \bS$ and $\bS_k = \calT_1( \calT_2(\bS_{k-1}))$ for $k\geq 1$.
Transformations $\calT_1$ and $\calT_2$  decrease $m(\calG)$ or $m(\calH)$, unless they act as the identity transformation, by \eqref{o4.3} and \eqref{o4.2}. Hence, 
for some $k_1$ and all $k\geq k_1$, $\bS_k = \bS_{k_1}$.
We let $\bS'= \bS_{k_1}$ and note that $\calT_1(\bS') = \calT_2(\bS') = \bS'$. This implies that $\bS'$ must satisfy the properties listed in Lemmas \ref{s29.2} and \ref{s29.3}, i.e.,  \eqref{f6.8}-\eqref{f3.12}, and the property that there is no cell
$C_{k,i} \subset B$ that is adjacent to  a cell  in $B^c$ and such that $\P(C_{k,i})=0$. A different way of saying this is that $\P(C'_{i,j}) >0$ for all $(i,j) \in\calD'_- \cup\calD'_+$. 

Conditions \eqref{f3.1}-\eqref{o2.1} hold for the transformation defined in this proof because they hold for $\calT_1$ and $\calT_2$. 
\end{proof}

\begin{lemma} \label{s29.8}
Suppose that $\bS$ satisfies conditions \eqref{f6.8}-\eqref{f3.12}. There exists $\bS''$ such that \eqref{f3.1}-\eqref{o2.1} and \eqref{f6.8}-\eqref{f3.12} hold, and
\begin{align}\label{s29.4}
\P(C''_{k,j}) >
\P(C''_{k,j}\cap  (A'')^c) =0 \quad \text{  or  }\quad 
\P(C''_{k,j}) > \P(C''_{k,j}\cap  A'') =0,
\end{align}
for all $(k,j) \in\calD''_- \cup \calD''_+$.
\end{lemma}

We used double primes in the statement of the lemma so that $\bS''$ is not confused with $\bS'$ constructed in the first step of the proof.

\begin{proof}[Proof of Lemma \ref{s29.8}]
The main idea of the proof is to fix  $(k,j) \in\calD_-\cup \calD_+$ and  move as much as possible of $A^c$ inside $C_{k,j}$ to $A$ without destroying those properties of $\bS$ that need to be preserved. If this is not possible, a similar transformation is applied to $A$ in place of $A^c$. Then the transformation is repeatedly  applied to all $(k,j) \in\calD_-\cup \calD_+$.

First we apply the transformation defined in Lemma \ref{o3.7}, if necessary, so that we can assume that  \eqref{o1.1} and \eqref{f6.8}-\eqref{f3.12} are satisfied for $\bS$, and  $\P(C_{i,j}) >0$ for all $(i,j) \in\calD_- \cup\calD_+$.

\medskip
\emph{Step 1}.
Consider any $(k,j) \in\calD_- \cup \calD_+$. If 
\begin{align}\label{s29.7}
\P(C_{k,j}\cap  (A)^c) =0 \quad \text{  or  }\quad \P(C_{k,j}\cap  A) =0
\end{align}
then we let $\bS' =\bS$ and we let $\calT_3(k,j)$ act as the identity transformation, i.e., $\calT_3(k,j)(\bS)=\bS$. In this case at least one of the conditions in \eqref{f6.15} (see below) holds. 

Otherwise we proceed as follows.
Assume that $(k,j) \in\calD_-$ and  
\begin{align}\label{j26.5}
 p_{k} \geq  q_{j}.
\end{align}
Let $\alpha\geq 0$  be the largest real number such that the following three conditions are satisfied,
\begin{align}\label{j23.1}
\frac{\alpha }
{ p_k }+ x_{k} &\leq  x_{k+1}\land \delta,\\
\label{j23.2}
\frac{\alpha }
{ q_j } + y_{j} &\leq  y_{j+1} \land 1,\\
\alpha & \leq \P(C_{k,j}\cap A^c).\label{f6.7}
\end{align}
Here, if $k = m(\calG)$ then, by convention, $x_{k+1}=\infty$ and, similarly,
if $j = m(\calH)$ then $y_{j+1}=\infty$.
At least one of the inequalities \eqref{j23.1}-\eqref{f6.7} is the equality. 

We make $\calF$ finer, if necessary (see Lemma \ref{s29.5}), and find an event $A_* \subset C_{k,j}\cap A^c$ such that $\P(A_*) = \alpha$. Then we let $A'= A\cup A_*$. The other elements of $\bS$ remain unchanged.
We have assumed that \eqref{s29.7} does not hold so
\begin{align}\label{o4.5}
\P(C'_{k,j}) \geq \P(C'_{k,j}\cap A')\geq \P(C_{k,j}\cap A)>0.
\end{align}
We have
\begin{align}\label{j22.1}
x_{k}' &= \frac{\alpha }
{ p_k }+ x_{k},\\
x'_i &= x_i, \qquad \text{  for  } i\ne k,\label{j26.8}\\
y'_{j} &= \frac{\alpha }
{ q_j } + y_{j},\label{j26.6}\\
y'_i &= y_i,\qquad  \text{  for  }  i\ne j.\label{j25.2}
\end{align}

Clearly, condition \eqref{s28.6} is satisfied.
To prove \eqref{f3.1} and \eqref{o1.1}-\eqref{o2.1}, it will suffice to show that for all $i$ and $n$,
\begin{align}\label{j26.3}
\text{  If } \quad |x_n - y_i|\geq 1-\delta \quad 
\text{  then  }  \quad |x'_n - y'_i|\geq 1-\delta .
\end{align}
In view of \eqref{j22.1}-\eqref{j25.2}, \eqref{j26.3} holds if
$n\ne k$ and $i\ne j$.

Conditions \eqref{f3.12}, \eqref{f6.1}-\eqref{f6.3} and the assumption that $(k,j) \in\calD_-$ imply that
 $|x_k - y_i|\geq 1-\delta$ if and only if 
$i\geq j$. It follows from  \eqref{j26.5}, \eqref{j22.1} and \eqref{j26.6}-\eqref{j25.2} that if $i \geq j$ then
\begin{align*}
y'_{i} - x'_k \geq y'_{j} - x'_k 
= 
\frac{\alpha }
{ q_j } + y_{j}
- \frac{\alpha }
{ p_k }- x_{k} 
\geq   y_{j}
-  x_{k}\geq 1-\delta.
\end{align*}
This proves \eqref{j26.3} for $n=k$ and all $i$.

Conditions \eqref{f3.12}, \eqref{f6.1}-\eqref{f6.3} and the assumption that $(k,j) \in\calD_-$ imply that
 $|x_n - y_j|\geq 1-\delta$ if and only if 
$n\leq k$. It follows from  \eqref{j26.5}, \eqref{j23.1},  \eqref{j22.1}-\eqref{j26.8} and \eqref{j26.6} that if $n\leq k$ then
\begin{align*}
y'_{j} - x'_n \geq y'_{j} - x'_k
= 
\frac{\alpha }
{ q_j } + y_{j}
-  \frac{\alpha }
{ p_k } - x_{k} 
\geq   y_{j}
-  x_{k}\geq 1-\delta.
\end{align*}
This completes the proof of \eqref{j26.3} for $i=j$ and all $n$. Hence, \eqref{f3.1} and \eqref{o1.1}-\eqref{o2.1} hold true.

Recall that
at least one of the inequalities \eqref{j23.1}-\eqref{f6.7} is the equality. 
We list all possible cases below.

If $x'_k={\alpha }/
{ p_k }+ x_{k} = \delta$ then we must have $y'_j=1$ in view of \eqref{j26.3}. This implies that $\P(C'_{k,j}\cap  (A')^c) =0$.

If $x'_k={\alpha }/
{ p_k }+ x_{k} = x_{k+1}<\delta$ then $x'_k = x'_{k+1}$.

If $y'_j={\alpha }/{ q_j } + y_{j} = 1$ then  $\P(C'_{k,j}\cap  (A')^c) =0$.

If $y'_j={\alpha }/{ q_j } + y_{j} =y_{j+1}< 1$ then  $y'_j = y'_{j+1}$.

If $\alpha = \P(C_{k,j}\cap A^c)$ then $\P(C'_{k,j}\cap  (A')^c) =0$.

These observations and \eqref{o4.5} allow us to conclude that  $x'_k = x'_{k+1}$ or  $y'_j = y'_{j+1}$ or $\P(C'_{k,j})>\P(C'_{k,j}\cap  (A')^c) =0$.

A completely analogous argument shows that if \eqref{j26.5} does not hold then we can transform $\bS$ into $\bS'$ such that \eqref{f3.1}-\eqref{o2.1} are true and we have the alternative: $x'_k = x'_{k-1}$ or  $y'_j = y'_{j-1}$ or $\P(C'_{k,j})>\P(C'_{k,j}\cap  A') =0$.

We next replace the assumption that $(k,j) \in\calD_-$ with $(k,j) \in\calD_+$. Once again, we can apply the analogous argument to conclude that 
we can construct $\bS'$ satisfying  \eqref{f3.1}-\eqref{o2.1} and such that 
$x'_k = x'_{k+1}$ or  $y'_j = y'_{j+1}$ or $\P(C'_{k,j})>\P(C'_{k,j}\cap  (A')^c) =0$ or $x'_k = x'_{k-1}$ or  $y'_j = y'_{j-1}$ or $\P(C'_{k,j})>\P(C'_{k,j}\cap  A') =0$.

We see that if $(k,j) \in\calD_- \cup \calD_+$ then
\begin{align}\label{f6.14}
x'_k = x'_{k+1} \quad \text{  or  }\quad y'_j = y'_{j+1} \quad \text{  or  }\quad x'_k = x'_{k-1} \quad \text{  or  }\quad y'_j = y'_{j-1}
\end{align}
or
\begin{align}\label{f6.15}
\P(C'_{k,j})>\P(C'_{k,j}\cap  (A')^c) =0 \quad \text{  or  }\quad \P(C'_{k,j})>\P(C'_{k,j}\cap  A') =0.
\end{align}
Let $\calT_3(k,j)$ be the transformation defined in this step.

\medskip
\emph{Step 2}.
Let $\prec$ denote an arbitrary order for the set of pairs $(k,j)$.
Let $\calT_3$ be defined as $\calT_3(k,j)$ applied to the first pair $(k,j)$ (in the sense of the order $\prec$) such that \eqref{s29.7} does not hold.

Let $\calT_4$ denote the transformation defined in Lemma \ref{s29.2}.

Let $\bS_0 = \bS$ and $\bS_k = \calT_3(\calT_4(\bS_{k-1}))$ for $k\geq 1$.
 The transformation $\calT_4$ strictly decreases $m(\calG)$ or $m(\calH)$, unless it acts as the identity transformation. Hence, for some $n_1$ and all $k\geq n_1$, $\bS_k = \calT_3(\bS_{k-1})$. This implies that for $k\geq n_1$, $\bS_k$ satisfies \eqref{f3.12}, and this in turn implies that \eqref{f6.14} cannot be true for $\bS_k$.

Since none of the conditions in \eqref{f6.14} can hold, one of the conditions in \eqref{f6.15} must be true for $\bS_k$, $k\geq n_1$. Note that $\calT_3(k,j)$ does not change $C_{n,i} \cap A$ or $C_{n,i} \cap A^c$ unless $(n,i) = (k,j)$.
Hence, for some $n_2 \geq n_1$ and all $k\geq n_2$, $\bS_k = \calT_3(\bS_{n_2})$. We let $\bS''=\bS_{n_2}$. 

Conditions \eqref{f3.1}-\eqref{o2.1} hold for the transformation defined in this proof because they hold for $\calT_3$ and $\calT_4$. 
\end{proof}

Consider four cells that lie at the corners of a rectangle, i.e., 
cells $C_{k_1,j_1}, C_{k_2,j_2} , C_{k_1,j_2}$ and $  C_{k_2,j_1}$ for some $j_1,j_2,k_1$ and $k_2$. The  transformation defined in the next lemma moves as much as possible of $A$ from the  upper left  corner to the  lower left corner, and compensates by moving the same amount of $A$ from the lower right corner to the upper right corner. 
The result of the transformation is that one of the four cells will not hold any $A$.

\begin{lemma}\label{s30.5}
Suppose that for some $\bS$, 
$1\leq k_1 < k_2 \leq m(\calG)$ and $1 \leq j_2 < j_1 \leq m(\calH)$ we have
\begin{align}
&p := \min(
\P(A \cap C_{k_1,j_1} ) , \P(A \cap C_{k_2,j_2} )) >0.\label{f7.11}
\end{align}
Assume that
\begin{align}\label{f7.2}
C_{k_1,j_1} \cup C_{k_2,j_2} &\cup C_{k_1,j_2} \cup  C_{k_2,j_1}
\subset B, \quad \text{  or  }\\
\label{f7.3}
C_{k_1,j_1} \cup C_{k_1,j_2} 
\subset B \quad &\text{  and  } \quad
 C_{k_2,j_1} \cup  C_{k_2,j_2}
\subset B^c, \quad \text{  or  }\\
C_{k_1,j_1} \cup C_{k_1,j_2} 
\subset B^c \quad &\text{  and  } \quad
 C_{k_2,j_1} \cup  C_{k_2,j_2}
\subset B, \quad \text{  or  }\label{f8.1}\\
C_{k_1,j_1} \cup C_{k_2,j_1} 
\subset B \quad &\text{  and  } \quad
 C_{k_1,j_2} \cup  C_{k_2,j_2}
\subset B^c, \quad \text{  or  }\label{f8.2}\\
C_{k_1,j_1} \cup C_{k_2,j_1} 
\subset B^c \quad &\text{  and  } \quad
 C_{k_1,j_2} \cup  C_{k_2,j_2}
\subset B.\label{f8.3}
\end{align}

There exists $\bS'$ such that $m'(\calG') = m(\calG)$, $m'(\calH') = m(\calH)$,
\begin{align}
\P(C'_{k_1,j_1})&= \P(C_{k_1,j_1}) -p,\label{s30.1}\\
\P(C'_{k_2,j_2})&= \P(C_{k_2,j_2}) -p,\label{s30.2}\\
\P(C'_{k_1,j_2})&= \P(C_{k_1,j_2}) +p,\label{s30.3}\\
\P(C'_{k_2,j_1})&= \P(C_{k_2,j_1}) +p,\label{s30.4}
\end{align}
and $\P(C'_{k,j}) = \P(C_{k,j})$
for all other values of $(k,j)$. 

Moreover,  \eqref{o1.1}-\eqref{o2.1} hold.
Condition \eqref{f3.1} holds with equality.
\end{lemma}

\begin{proof}
We can make $\calF$ finer, if necessary (see Lemma \ref{s29.5}), so that there exist sets $A_1 \subset A \cap C_{k_1,j_1} $ and $A_2 \subset A \cap C_{k_2,j_2}$ such that
$\P(A_1) = \P(A_2) = p$. Let 
\begin{align*}
C'_{k_1,j_1}&= C_{k_1,j_1} \setminus A_1,\qquad
C'_{k_2,j_2}= C_{k_2,j_2} \setminus A_2,\\
C'_{k_1,j_2}&= C_{k_1,j_2} \cup A_1,\qquad
C'_{k_2,j_1}= C_{k_2,j_1} \cup A_2.
\end{align*}
For all other values of $(k,j)$, we let $C'_{k,j} = C_{k,j}$. These transformations redefine $G_k$'s, $H_j$'s, $\calG$ and $\calH$. Other elements of $\bS$ are unaffected.

It is clear that \eqref{s30.1}-\eqref{s30.4} are satisfied and $\P(C'_{k,j}) = \P(C_{k,j})$
for all other values of $(k,j)$.

It is easy to check that $p'_k = p_k$, $q'_k = q_k$, $x'_k = x_k$ and $y'_k=y_k$ for all $k$. Therefore, $C'_{k,j} \subset B'$ if and only if $C_{k,j} \subset B$, for every $(k,j)$.

Suppose that \eqref{f7.3} is true. Then
\begin{align}
\P(C'_{k_1,j_1} \cap B') + 
\P( C'_{k_1,j_2} \cap B')
&= \P(C'_{k_1,j_1} ) + 
\P( C'_{k_1,j_2} )
= \P(C_{k_1,j_1} )-p + 
\P( C_{k_1,j_2} )+p \notag \\
&= \P(C_{k_1,j_1} ) + 
\P( C_{k_1,j_2} )
=\P(C_{k_1,j_1} \cap B) + 
\P( C_{k_1,j_2} \cap B),\label{o3.8}
\end{align}
and
\begin{align*}
\P(C'_{k_2,j_1} \cap B') + 
\P( C'_{k_2,j_2} \cap B')
&= 0
=\P(C_{k_2,j_1} \cap B) + 
\P( C_{k_2,j_2} \cap B).
\end{align*}
These formulas and the fact that $\P(C'_{k,j}) = \P(C_{k,j})$
for all other values of $(k,j)$ imply that \eqref{f3.1} holds with equality. 

It follows from \eqref{f7.3} that either $j_1,j_2 \leq m_-(\calH)$ or $j_1,j_2 \geq m_+(\calH)$. This and \eqref{o3.8} imply that \eqref{o1.1}-\eqref{o2.1} holds.

One can prove that \eqref{o1.1}-\eqref{o2.1} holds and \eqref{f3.1} holds with equality under any of the assumptions \eqref{f7.2}-\eqref{f8.3} in a similar manner. 

The condition \eqref{s28.6} obviously holds.
\end{proof}

\begin{notation}
We will use $\tnw((k_1,j_1),(k_2,j_2))$ to denote 
the transformation defined in Lemma \ref{s30.5}.
We define a transformation $\tnwc((k_1,j_1),(k_2,j_2))$ by replacing
 $A$ with $A^c$ in the assumption \eqref{f7.11} and the construction of $\tnw((k_1,j_1),(k_2,j_2))$.

If  the assumptions of Lemma \ref{s30.5} do not hold then $\tnw((k_1,j_1),(k_2,j_2))$ will denote the identity transformation, and similarly for $\tnwc((k_1,j_1),(k_2,j_2))$.
\end{notation}

The transformation defined in the next lemma converts the top right family of cells in  Fig. \ref{fig1} into subsets of $A$. It also converts the bottom left family of cells in  Fig. \ref{fig1} into subsets of $A^c$.

\begin{lemma}\label{s30.10}
Given $\bS$, let
\begin{align}\label{f14.3}
A' &=\left( A\cup 
\bigcup_{m_-(\calG) < k \leq m(\calG), 
m_-(\calH) < j \leq m(\calH)} C_{k,j} \right)
\setminus 
\left(\bigcup_{1 \leq k < m_+(\calG), 
1 \leq j < m_+(\calH)} C_{k,j}\right).
\end{align}
Let all other elements of $\bS'$ be the same as those of $\bS$.
Then conditions \eqref{f3.1}-\eqref{o2.1} hold.
\end{lemma}

\begin{proof}
The condition \eqref{s28.6} clearly holds.

Note that 
\begin{align*}
x'_k & \leq x_k, \quad \text{  for  } 1 \leq k \leq m_-(\calG),\\
y'_k & \leq y_k, \quad \text{  for  } 1 \leq k \leq m_-(\calH),\\
x'_k & \geq x_k, \quad \text{  for  } m_+(\calG) \leq k \leq m(\calG),\\
y'_k & \geq y_k, \quad \text{  for  } m_+(\calH) \leq k \leq m(\calH).
\end{align*}
These observations and \eqref{f7.4}-\eqref{f7.7} imply that $B\subset B'$ and, therefore,
\eqref{f3.1} and \eqref{o1.1}-\eqref{o2.1} hold.
\end{proof}

The transformation defined in the next lemma empties all cells in the ``bottom right corner'' of the ``upper left corner'' (marked with the green color in  Fig. \ref{fig1}), and similarly on the other side of the diagonal.
The result of the transformations described
in Lemmas \ref{s30.10}-\ref{s30.11}  is that large regions in Fig. \ref{fig1} do not have cells that would hold both $A$ and $A^c$.

\begin{lemma}\label{s30.11}
Consider an $\bS$ such that \eqref{o1.1} holds. Then there exists $\bS'$ such that conditions \eqref{f3.1}-\eqref{o2.1}, \eqref{f6.8}-\eqref{f3.12}  and \eqref{s29.4} hold, and
\begin{align}\label{f9.2}
\P((G'_k \setminus B') \cap A') &=0 \quad \text{  for  } \quad 1\leq k \leq m_-(\calG'),\\
\P((H'_k \setminus B') \cap A') &=0 \quad \text{  for  } \quad 1\leq k \leq m_-(\calH'),\label{f9.3}\\
\P((G'_k \setminus B') \cap (A')^c) &=0 \quad \text{  for  } \quad m_+(\calG) \leq k \leq m(\calG),\label{f9.4}\\
\P((H'_k \setminus B') \cap (A')^c) &=0\quad \text{  for  } \quad m_+(\calH) \leq k \leq m(\calH).\label{f9.5}
\end{align}
\end{lemma}

\begin{proof}
\emph{Step 1}.
It follows from \eqref{o1.1} that $ m_-(\calG)\geq 1$, $ m_+(\calH)\leq m(\calH)$, $ m_-(\calH)\geq 1$, and $ m_+(\calG)\leq m(\calG)$.
Consider $k$ and $j$ such that $k \leq m_-(\calG)$, $j\geq m_+(\calH) $ and  $C_{k,j} \subset B^c$. Let $C'_{k,j} = \emptyset$.

If $\P(A\mid C_{k,j}) < 1-\delta$ then let $ A' = A \setminus C_{k,j}$ and $C'_{k,1} = C_{k,1} \cup C_{k,j}$.

If $\P(A\mid C_{k,j}) \geq 1-\delta$ then let $ A' = A \cup C_{k,j}$ and $C'_{m(\calG),j} = C_{m(\calG),j} \cup C_{k,j}$.

These changes will affect $\calG$ and $\calH$. Other elements of $\bS$ will be unchanged.

It is easy to check that $x'_k \leq x_k$, $x'_{m(\calG)} \geq x_{m(\calG)}$, $y'_1 \leq y_1$  and $y'_j \geq y_j$. All other $x_i$'s and $y_i$'s will be unaffected.
For all $i$ and $n$ such that $C_{i,n}\subset B$, $\P(C'_{i,n}) = \P(C_{i,n})$.
Hence, in view of \eqref{f7.4}-\eqref{f7.7}, we see that
\eqref{f3.1}-\eqref{o2.1} hold.
We also have
\begin{align}\label{o6.1}
A \cap 
\bigcup_{m_-(\calG) < k \leq m(\calG), 
m_-(\calH) < j \leq m(\calH)} C_{k,j} 
&\subset 
A' \cap 
\bigcup_{m'_-(\calG') < k \leq m'(\calG'), 
m'_-(\calH') < j \leq m'(\calH')} C'_{k,j} ,\\
A' \cap \bigcup_{1 \leq k < m'_+(\calG'), 
1 \leq j < m'_+(\calH')} C'_{k,j}
&\subset
A \cap \bigcup_{1 \leq k < m_+(\calG), 
1 \leq j < m_+(\calH)} C_{k,j}.\label{o6.2}
\end{align}

We repeat the transformation for all $k$ and $j$ such that $k \leq m_-(\calG)$, $j\geq m_+(\calH) $ and  $C_{k,j} \subset B^c$. The result  is that $C'_{k,j}=\emptyset$ for all such pairs $(k,j)$. Moreover, \eqref{f3.1}-\eqref{o2.1} and \eqref{o6.1}-\eqref{o6.2} hold.

We apply an analogous sequence of transformations corresponding to all pairs $(k,j)$ such that $k \geq m_+(\calG)$, $j\leq m_-(\calH) $ and  $C_{k,j} \subset B^c$. The result is, once again, that $C'_{k,j}=\emptyset$ for all such pairs $(k,j)$. The conditions \eqref{f3.1}-\eqref{o2.1} and \eqref{o6.1}-\eqref{o6.2} still hold. Note that $A$ and $A'$ do not exchange the roles in \eqref{o6.1}-\eqref{o6.2}, i.e., these conditions hold the way they are stated.

 We will denote the composition of all transformations defined in this step by $\calT_5$. 

For later reference, we record the following properties of $\bS' = \calT_5(\bS)$,
\begin{align}\label{f14.1}
C'_{k,j}&=\emptyset \quad\text{  if  } k \leq m_-(\calG),\  j\geq m_+(\calH) \text{  and } C'_{k,j} \subset B^c,\\
C'_{k,j}&=\emptyset \quad\text{  if  } k \geq m_+(\calG),\  j\leq m_-(\calH) \text{  and  }  C'_{k,j} \subset B^c. \label{f14.2}
\end{align}

\medskip
\emph{Step 2}.
Let $\calT_6, \calT_7$ and $\calT_8$ denote the transformations defined in Lemmas \ref{s30.10}, \ref{s29.8} and \ref{s29.2}. 
Let $\bS_0=\bS$ and $\bS_k = \calT_5(\calT_6(\calT_7(\calT_8(\bS_{k-1}))))$ for $k\geq 1$. In view of \eqref{o4.3} and the fact that $\calT_5, \calT_6$ and $\calT_7$ have the property \eqref{s28.6}, $\calT_8$ must act as an identity eventually, i.e., there must exist $n_1$ such that $\bS_k = \calT_5(\calT_6(\calT_7(\bS_{k-1})))$ for $k\geq n_1$.

Transformations $\calT_5$ and $\calT_6$ do not change $m(\calG)$ and $m(\calH)$. They also do not change the values of $\P(C_{k,j})$ and $\P(C_{j,k}\cap A)$ for $(k,j) \in \calD_- \cap \calD_+$. It follows that $\calT_7$ must act as an identity eventually, i.e., there must exist $n_2\geq n_1$ such that $\bS_k = \calT_5(\calT_6(\bS_{k-1}))$ for $k\geq n_2$.

It follows from \eqref{f14.3} and \eqref{o6.1}-\eqref{o6.2} that  there must exist $n_3\geq n_2$ such that $\bS_k = \bS_{k-1}$ for $k\geq n_3$.
Let $\bS'= \bS_{n_3}$. Since $\calT_i(\bS') = \bS'$ for $i=5,6,7,8$, the conditions  \eqref{f3.1}-\eqref{o2.1}, \eqref{f6.8}-\eqref{f3.12}, \eqref{s29.4} and \eqref{f9.2}-\eqref{f9.5} are satisfied.  
\end{proof}

\begin{lemma}\label{o3.10}
If $\P(B)>0$ and $\eps>0$ then there exists $\bS'$ satisfying 

(i) $\P(B') > \P(B) - \eps$,

(ii) $m_-(\calG) = 1$, or  $m_-(\calG) =2$ and   $\P(C_{1,m(\calH)})=0$, 

(iii) $m_-(\calH) = 1$, or  $m_-(\calH) =2$ and   $\P(C_{m(\calG),1})=0$.
\end{lemma}

\begin{proof}
Suppose that $\P(B)>0$ and $\eps>0$. We apply the transformation defined in Lemma \ref{o1.2} and obtain $\bS'$ satisfying $\P(B') > \P(B) - \eps$, $m'(\calG') \leq m(\calG)+1$,  $m'(\calH') \leq m(\calH)+1$,  and 
\eqref{o2.1}.

Next we apply the transformation defined in Lemma \ref{s30.11} and obtain 
$\bS''$ satisfying \eqref{f3.1}, \eqref{o2.1}, \eqref{f6.8}-\eqref{f3.12}, \eqref{s29.4}, \eqref{f9.2}-\eqref{f9.5}, $m''(\calG'') \leq m(\calG)+1$ and $m''(\calH'') \leq m(\calH)+1$.  
These properties of $\bS''$ are illustrated in Fig.~\ref{fig1}.

\begin{figure} \includegraphics[width=1.0\linewidth]{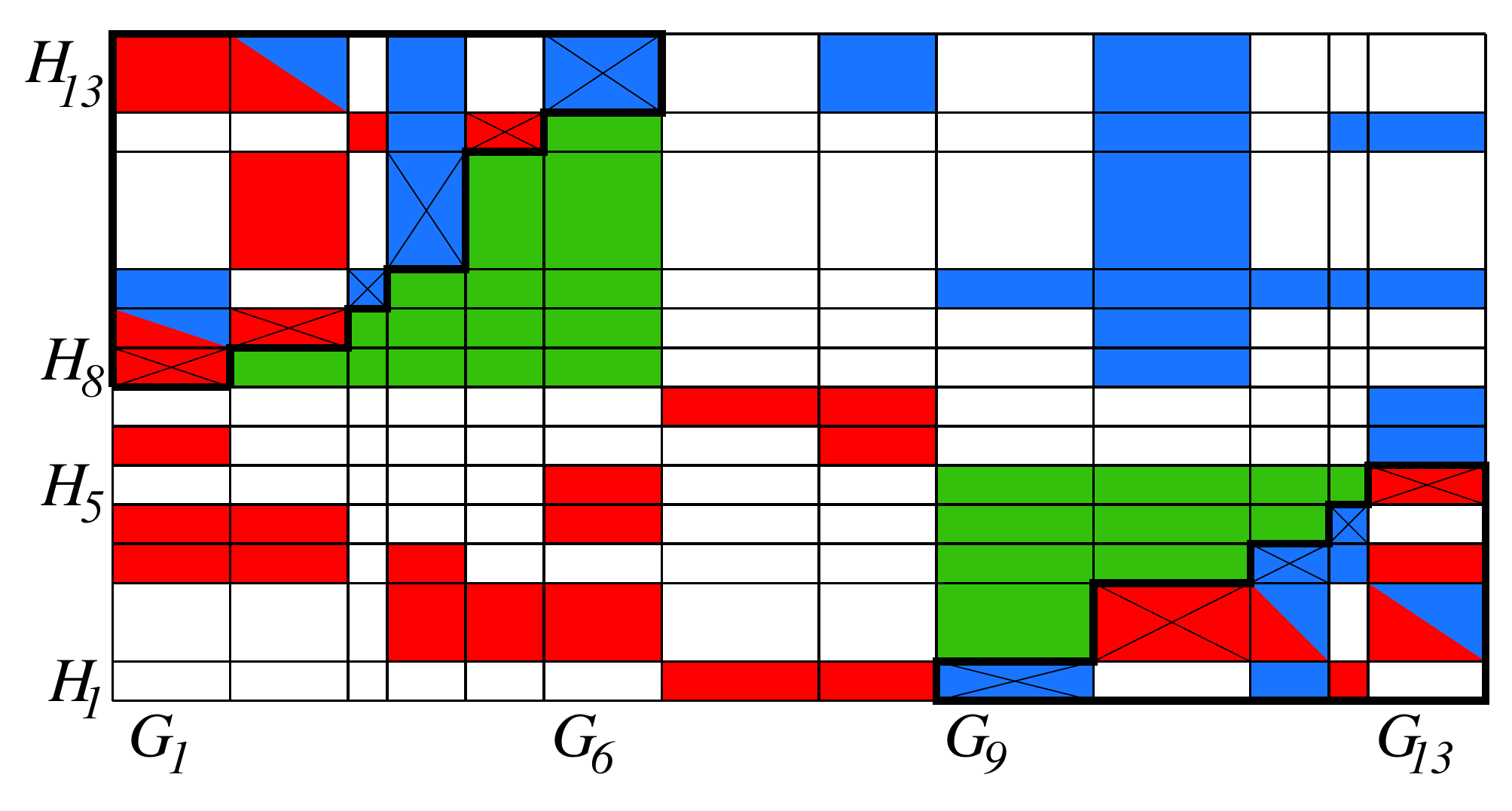}
\caption{ (Color coded)
For the meaning of white, red and blue colors and thick lines, see the caption of Fig.~\ref{fig10}.
Green color  has the same meaning as white color, i.e., it represents ``empty'' cells, i.e., cells $C_{k,j}$ such that $\P(C_{k,j})=0$.
The green cells are the cells that were ``emptied'' in Step 1 of the proof of Lemma \ref{s30.11}.
Every crossed cell belongs either to $\calD_- $ or to $ \calD_+$ and must
satisfy $\P(C_{k,j} )>0$ and either $\P(C_{k,j} \cap A) =0$ or $\P(C_{k,j} \cap A^c) =0$. 
All cells below $\calD_-$ are in $A^c$ or empty. All cells above $\calD_+$ are in $A$ or empty. All cells to the right of $\calD_-$ are in $A$ or empty. All cells to the left of $\calD_+$ are in $A^c$ or empty. 
If a cell is below a cell in $\calD_-$ and to the right of a (different) cell in $\calD_-$ then it is empty. If a cell is above a cell in $\calD_+$ and to the left of a (different) cell in $\calD_+$ then it is empty.
}
\label{fig1}
\end{figure}

The logical scheme of the remaining part of the proof is  represented by the flowchart in Fig.~\ref{Ltree}. 

\begin{figure} \includegraphics[width=0.8\linewidth]{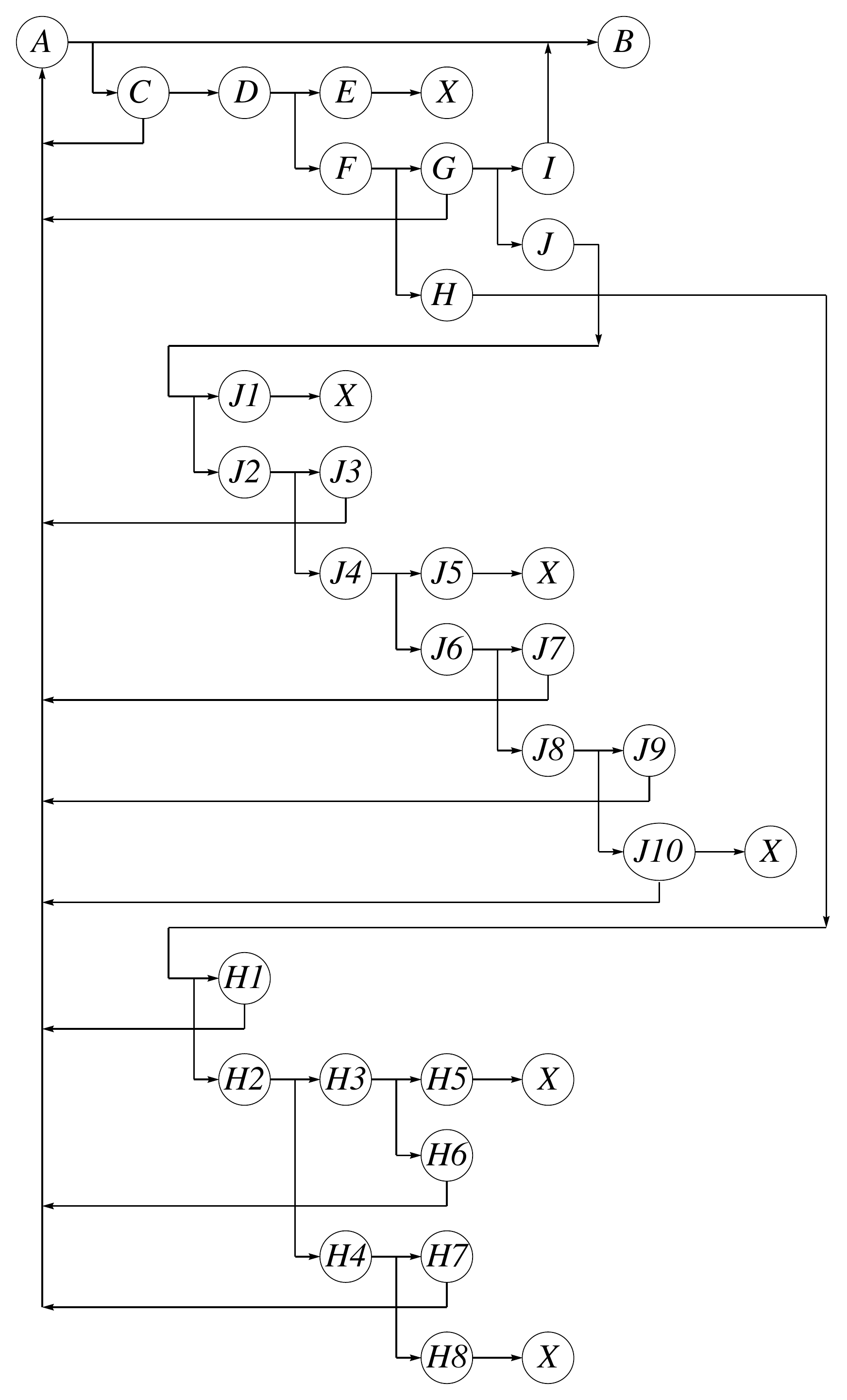}
\caption{
The logical structure of the proof of Lemma \ref{o3.10}.
}
\label{Ltree}
\end{figure}

The leaves of our branching argument will be denoted  \circled{X}. These are places where logical contradictions are reached so another branch of the proof must be considered.

We will now give names to transformations that we will apply in this proof.

Let $\calT_9$ be the transformation defined in Lemma \ref{s30.11}. 

Recall that $\tnw((k_1,j_1),(k_2,j_2))$  denotes
the transformation defined in Lemma \ref{s30.5}.
The transformation $\tnwc((k_1,j_1),(k_2,j_2))$ was defined in an analogous way, by replacing $A$ with $A^c$ in the assumption \eqref{f7.11} and the construction of $\tnw((k_1,j_1),(k_2,j_2))$.

Recall that $\calT_3(k,j)$ is the transformation defined in Step 1 of the proof of Lemma \ref{s29.8}. If $\bS' = \calT_3(k,j)(\bS)$ then at least one of the conditions \eqref{f6.14}-\eqref{f6.15} holds.

The root of the flowchart  in Fig.~\ref{Ltree} is \circled{A}.
Our argument will require that we jump to \circled{A} repeatedly. We will argue in the main body of the proof that every jump to \circled{A} is associated with the decrease of $m(\calG)$ or $m(\calH)$. Therefore, there can be only a finite number of jumps to \circled{A} from later parts of the proof. 

We will apply transformations named above repeatedly but instead of using the notation with primes, such as $C'_{k,j}$, for the resulting objects, we will always write $C_{k,j}$, etc., without primes, to simplify the notation. We hope that this convention will not be confusing in this proof.

Suppose that $\bS$ is given. In view of the initial part of the proof, we can assume that it has all the properties of $\bS''$, i.e., it satisfies \eqref{f3.1}, \eqref{o2.1}, \eqref{f6.8}-\eqref{f3.12}, \eqref{s29.4}, \eqref{f9.2}-\eqref{f9.5}.

\circled{A} Apply $\calT_9$ to $\bS$. 

If $m_-(\calG) \leq 1$ then we jump to \circled{B}. 

\circled{C}
Suppose that $m_-(\calG) > 1$. 
We apply transformations $\tnwc((k,m(\calH)), (m_-(\calG), j))$ repeatedly for all  $k < m_-(\calG)$ and $j< m_+(\calH)$.
The resulting $\bS$ satisfies \eqref{f3.1} because the following case of \eqref{f7.3}-\eqref{f8.3} is satisfied for $k < m_-(\calG)$ and $j< m_+(\calH)$,
\begin{align*}
C_{k,m(\calH)} \cup C_{m_-(\calG),m(\calH)} 
\subset B \quad &\text{  and  } \quad
 C_{k,j} \cup  C_{m_-(\calG),j}
\subset B^c. 
\end{align*}
In later applications of $\tnw$ or $\tnwc$, we will leave it to the reader to check that one of conditions \eqref{f7.2}-\eqref{f8.3} is satisfied.

In view of \eqref{f14.1}-\eqref{f14.2}, the result of these transformations has the  property 
\begin{align}\label{f9.7}
\P\left(A^c \cap \bigcup _{k < m_-(\calG)} C_{k, m(\calH)} \right) &= 0,
\quad \text{  or  }\\
\P\left(A^c \cap \bigcup _{j < m(\calH)} C_{m_-(\calG), j} \right) &= 0.\label{f9.8}
\end{align}

We apply the transformation $\calT_3(m_-(\calG), m(\calH))$. Note that it changes the ``$A$-contents'' of only $C_{m_-(\calG), m(\calH)}$. The resulting $\bS$ satisfies \eqref{f6.14}  or \eqref{f6.15} with $(k,j) = (m_-(\calG), m(\calH))$.

If $\bS$ satisfies \eqref{f6.14} then we jump to \circled{A}.
Then an application of $\calT_9$ will result in the  decrease of $m(\calG)$ or $m(\calH)$, in view of \eqref{f3.12} and \eqref{o4.3}.

\circled{D}
Assume that \eqref{f6.15} holds, i.e., we have either $\P(C_{m_-(\calG),m(\calH)}\cap A) =0$ or \break
$\P(C_{m_-(\calG),m(\calH)}\cap A^c) =0$.

\circled{E}
Assume that \eqref{f9.8} holds and recall  \eqref{f9.2}-\eqref{f9.5}.
These formulas imply that $\P\left( \bigcup _{j < m(\calH)} C_{m_-(\calG), j} \right) = 0$.
Since $\P(C_{m_-(\calG),m(\calH)}) >0$ and either $\P(C_{m_-(\calG),m(\calH)}\cap A) =0$ or $\P(C_{m_-(\calG),m(\calH)}\cap A^c) =0$, we must have $x_{m_-(\calG)} = 0$ or $x_{m_-(\calG)} = 1$. If $x_{m_-(\calG)} =0$ then this contradicts the facts that $m_-(\calG)>1$ and $x_{m_-(\calG)} > x_1 \geq 0$. We cannot have $x_{m_-(\calG)} =1$ because  that would contradict \eqref{f7.4}.
We conclude that \eqref{f9.8} cannot hold.
\circled{X}

\circled{F} Assume that \eqref{f9.7} is true.
The transformations which generated \eqref{f9.7}-\eqref{f9.8} did not change $x_i$'s and $y_i$'s, and they also did not change the fact that $\P(C_{m_-(\calG),i}\cap A) =0$ 
for all $i< m(\calH)$.
We are in the current branch of the proof because we have not jumped to \circled{A} after applying 
 $\calT_3(m_-(\calG), m(\calH))$. Hence, $x_{m_-(\calG)} > x_{m_-(\calG)-1}$. 
If $\P(C_{m_-(\calG),m(\calH)}\cap A) =0$ then $x_{m_-(\calG)} =0$ but this contradicts the facts that $m_-(\calG)>1$ and $x_{m_-(\calG)} > x_1 \geq 0$. Hence, we must have  $\P(C_{m_-(\calG),m(\calH)}\cap A^c) =0$.
 This, \eqref{f9.2}-\eqref{f9.5}  and  \eqref{f9.7} imply that 
\begin{align}\label{f14.4}
\P(A^c \cap H_{m(\calH)})=\P\left(A^c \cap \bigcup _{1\leq k \leq m(\calG)} C_{k, m(\calH)} \right) = 0.
\end{align}

At this point our argument branches as follows.
We will consider the case $m_-(\calG)=2$ in \circled{G} and \circled{I}, the case
$m_-(\calG)=3$ in \circled{G} and \circled{J}, and the case $m_-(\calG)\geq 4$ in \circled{H}.
 
\circled{G}
If $m_-(\calG)=2$ or 3 then we  interchange the roles of $A$ and $A^c$, and $\calG$ and $\calH$, and argue as follows. 
We apply transformations $\tnw((k,m_+(\calH)), (1, j))$ repeatedly for all  $k > m_-(\calG)$ and $j> m_+(\calH)$. Then we apply the transformation $\calT_3(1, m_+(\calH))$.

If $\bS$ satisfies \eqref{f6.14} then we jump to \circled{A}. 
Then an application of $\calT_9$ will result in the  decrease of $m(\calG)$ or $m(\calH)$, in view of \eqref{f3.12} and \eqref{o4.3}.

Otherwise we will reach the conclusion analogous to \eqref{f14.4}, namely that 
\begin{align}\label{f14.5}
\P(A \cap G_{1})= 0.
\end{align}

We will now argue that transformations described between \eqref{f14.4} and \eqref{f14.5} do not affect the validity of \eqref{f14.4}, assuming that there was no jump to \circled{A}. Transformations $\tnw((k,m_+(\calH)), (1, j))$ do not change any $x_i$'s, $y_i$'s, $p_i$'s and $q_i$'s. The transformation $\calT_3(1, m_+(\calH))$ does not affect any cells in $H_{m(\calH)}$ because $m_+(\calH) < m(\calH)$. Hence, \eqref{f14.4} remains true.

\circled{I}
Suppose that $m_-(\calG)=2$ and use \eqref{f14.4}-\eqref{f14.5} to conclude that $\P(C_{1,m(\calH)})=0$.  We now jump to \circled{B}.
 
\circled{J}
Next assume that $m_-(\calG)=3$. This case is rather complicated so it is the only part of the proof whose  steps are illustrated one by one in Figs.~\ref{fig2}-\ref{fig8}. We will now explain how different events are represented in these figures. The three columns represent $G_1$, $G_2$ and $G_3$. The three rows represent $H_{m_+(\calH)} = H_{m(\calH)-2}$, $H_{m(\calH)-1}$ and $H_{m(\calH)}$.
The colors have the same meaning as in Fig.~\ref{fig1}.
Cells containing some white and some other color may be empty or contain either $A$ or $A^c$, depending on the color. The colored areas at the bottom represent the family of cells $C_{k,j}$ (not individual cells), with $1\leq k \leq 3$ and $j < m_+(\calH)$, and the colored areas on the right represent the family of cells $C_{k,j}$ with $k> m_-(\calG)$ and $H_{m_+(\calH)} \leq j \leq H_{m(\calH)}$.

The starting point is illustrated in Fig.~\ref{fig2}. This $\bS$ satisfies 
\eqref{f14.4}-\eqref{f14.5}. We have either
\begin{align}\label{f14.8}
\P(C_{2,m(\calH)-1}\cap A) =0
\end{align}
or $\P(C_{2,m(\calH)-1}\cap A^c) =0$. We will discuss only the  case in \eqref{f14.8}, depicted in Fig.~\ref{fig2}. The other case can be dealt with in an analogous way.

\begin{figure} \includegraphics[width=0.45\linewidth]{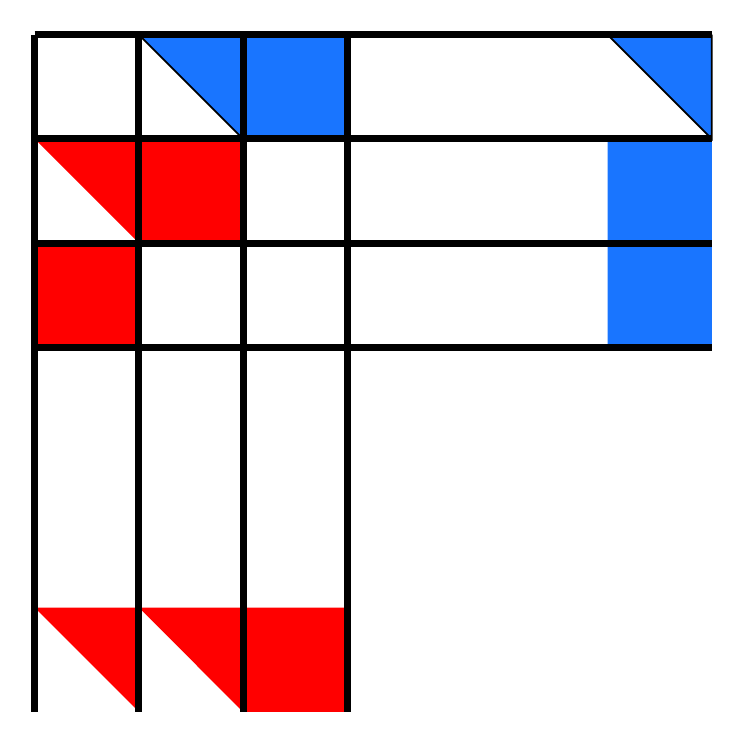}
\caption{(Color coded) See the body of the article for the description.}
\label{fig2}
\end{figure}

We apply transformations $\tnw((2,m(\calH)), (k,m(\calH)-1))$ repeatedly for all  $k>m_-(\calG)$.
The result of these transformations has the  property 
\begin{align}\label{f14.6}
\P\left(A \cap \bigcup _{k>m_-(\calG)} C_{k,m(\calH)-1)} \right) &= 0,
\quad \text{  or  }\\
\P\left(A \cap  C_{2,m(\calH)} \right) &= 0.\label{f14.7}
\end{align}

\circled{J1}
If \eqref{f14.6} holds then $\P(H_{m(\calH)-1} \cap A ) =0$ because of \eqref{f9.2}-\eqref{f9.5}, \eqref{f14.5} and \eqref{f14.8}. But then $y_{m(\calH)-1}=0$. This is impossible because $m(\calH)-1 > m_+(\calH)$. \circled{X}

\circled{J2}
Assume that \eqref{f14.7} is true. 
We apply  $\calT_3(2,m(\calH)-1)$. The resulting $\bS$ satisfies \eqref{f6.14}  or \eqref{f6.15} with $(k,j) = (2,m(\calH)-1)$.

\circled{J3}
If $\bS$ satisfies \eqref{f6.14} then we jump to \circled{A}. 
Then an application of $\calT_9$ will result in the  decrease of $m(\calG)$ or $m(\calH)$, in view of \eqref{f3.12} and \eqref{o4.3}.

\circled{J4}
Assume that \eqref{f6.15} holds, i.e., either $\P(C_{2,m(\calH)-1}\cap A) =0$ or $\P(C_{2,m(\calH)-1}\cap A^c) =0$ (see Figs.~\ref{fig3}-\ref{fig4}).

\circled{J5}
Assume that $\P(C_{2,m(\calH)-1}\cap A) =0$ (see Fig.~\ref{fig3}). This assumption, combined with \eqref{f9.2}-\eqref{f9.5} and \eqref{f14.7} implies that $x_2=\P(G_2 \cap A) = 0$. This is impossible because $x_2> x_1 \geq 0$. \circled{X}

\circled{J6} Hence, we must have  $\P(C_{2,m(\calH)-1}\cap A^c) =0$ (see Fig.~\ref{fig4}).
We apply transformations $\tnwc((1,m_+(\calH)+1), (2,j))$ repeatedly for all  $j<m_+(\calH)$.
The result of these transformations has the  property 
\begin{align}
\P\left(A^c \cap  C_{1,m_+(\calH)+1} \right) &= 0\label{f14.11}
\quad \text{  or  }\\
\label{f14.10}
\P\left(A^c \cap \bigcup _{j<m_+(\calH)} C_{2,j} \right) &= 0.
\end{align}
We apply the transformation $\calT_3(2,m(\calH)-1)$. The resulting $\bS$ satisfies \eqref{f6.14}  or \eqref{f6.15} with $(k,j) =(2,m(\calH)-1)$.

\circled{J7}
If $\bS$ satisfies \eqref{f6.14} then we jump to \circled{A}. 
Then an application of $\calT_9$ will result in the  decrease of $m(\calG)$ or $m(\calH)$, in view of \eqref{f3.12} and \eqref{o4.3}.

\begin{figure}
\centering
\begin{minipage}{.45\textwidth}
  \centering
  \includegraphics[width=1\linewidth]{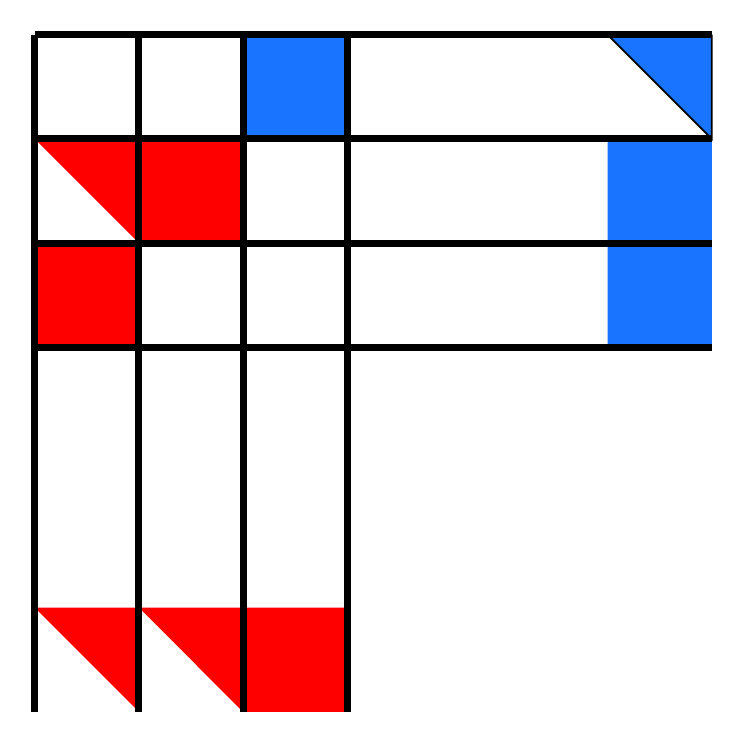}
  \captionof{figure}{}
  \label{fig3}
\end{minipage}%
\begin{minipage}{.45\textwidth}
  \centering
  \includegraphics[width=1\linewidth]{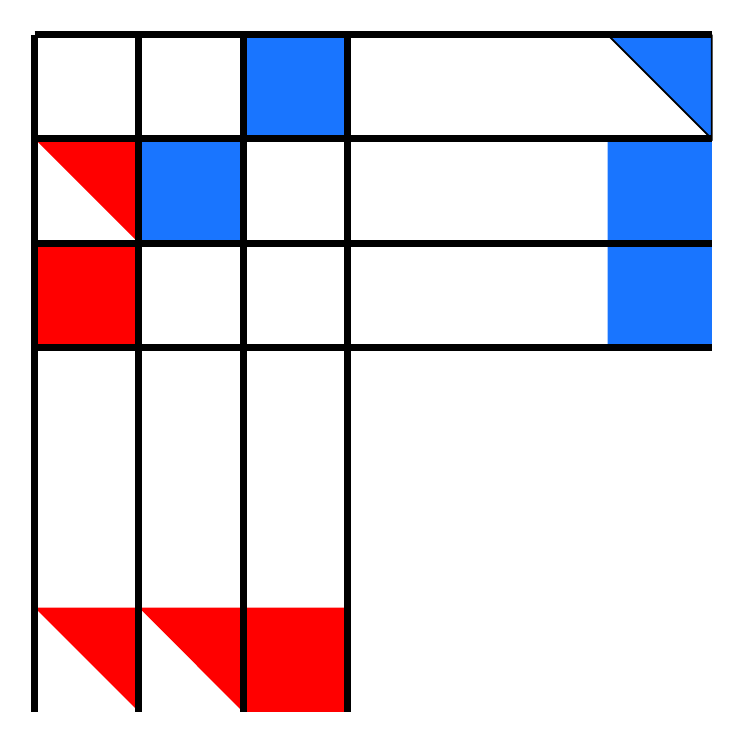}
  \captionof{figure}{}
  \label{fig4}
\end{minipage}
\captionsetup{labelformat=empty}
\caption{(Color coded figures) See the body of the article for the description.}
\end{figure}

\circled{J8}
Assume that \eqref{f6.15} holds (see Figs.~\ref{fig5}-\ref{fig6}), that is,
\begin{align}\label{f15.1}
\P(C_{2,m(\calH)-1}\cap A) =0\quad \text{  or  } \quad\P(C_{2,m(\calH)-1}\cap A^c) =0.
\end{align}

\circled{J9}
Suppose \eqref{f14.11} holds.

 If
$\P(C_{2,m(\calH)-1}\cap A^c) =0$ then $\P(H_{m(\calH)-1} \cap A^c ) =0$ because of \eqref{f9.2}-\eqref{f9.5} (see Fig.~\ref{fig5}). Then $y_{m(\calH)-1}=1= y_{m(\calH)}$ and we jump to \circled{A}. 
An application of $\calT_9$ will result in the  decrease of $m(\calG)$ or $m(\calH)$, in view of \eqref{f3.12} and \eqref{o4.3}.

If \eqref{f14.11} holds and $\P(C_{2,m(\calH)-1}\cap A) =0$ then $\P(G_2 \cap A ) =0$ because of \eqref{f9.2}-\eqref{f9.5} and \eqref{f14.7}  (see Fig.~\ref{fig6}). Then $x_2=0=x_1$  and we jump to \circled{A}.  
An application of $\calT_9$ will result in the  decrease of $m(\calG)$ or $m(\calH)$, in view of \eqref{f3.12} and \eqref{o4.3}.

\begin{figure}
\centering
\begin{minipage}{.45\textwidth}
  \centering
  \includegraphics[width=1\linewidth]{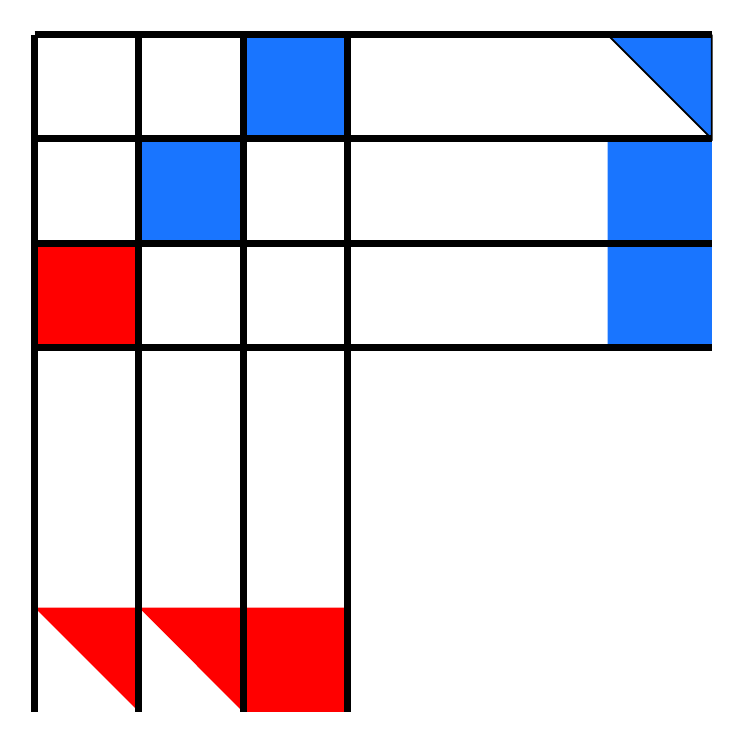}
  \captionof{figure}{}
  \label{fig5}
\end{minipage}%
\begin{minipage}{.45\textwidth}
  \centering
  \includegraphics[width=1\linewidth]{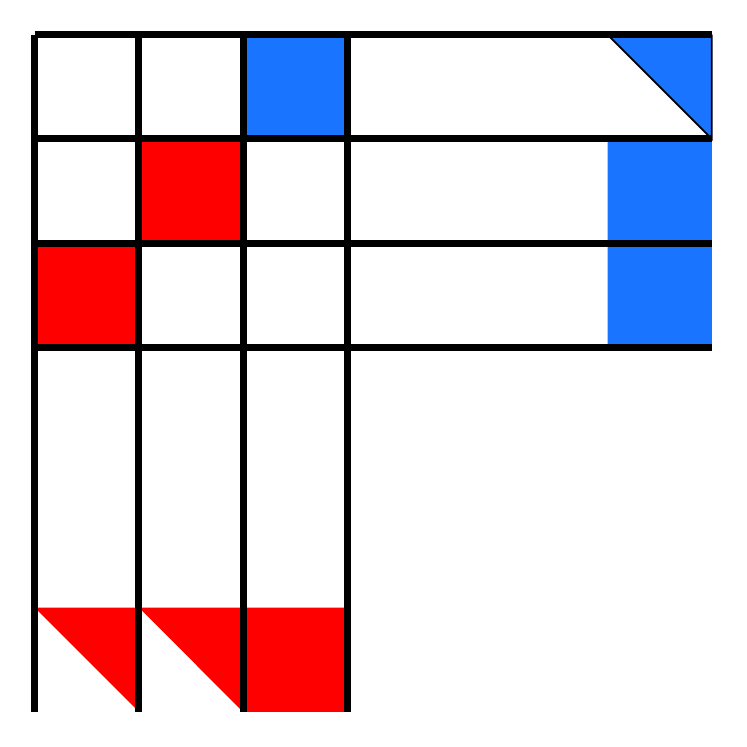}
  \captionof{figure}{}
  \label{fig6}
\end{minipage}
\captionsetup{labelformat=empty}
\caption{(Color coded figures) See the body of the article for the description.}
\end{figure}

\circled{J10}
Next suppose that \eqref{f14.10} holds; see Figs.~\ref{fig7}-\ref{fig8}.
In view of \eqref{f14.1}-\eqref{f14.2}, \eqref{f14.7} and \eqref{f15.1}, either $x_2 = 1$ or $x_2 =0$. The first case is impossible because $2< m_-(\calG)$. \circled{X}

 In the second case we have 
$x_2=0=x_1$  and we jump to \circled{A}. 
Then an application of $\calT_9$ will result in the  decrease of $m(\calG)$ or $m(\calH)$, in view of \eqref{f3.12} and \eqref{o4.3}.

\begin{figure}
\centering
\begin{minipage}{.45\textwidth}
  \centering
  \includegraphics[width=1\linewidth]{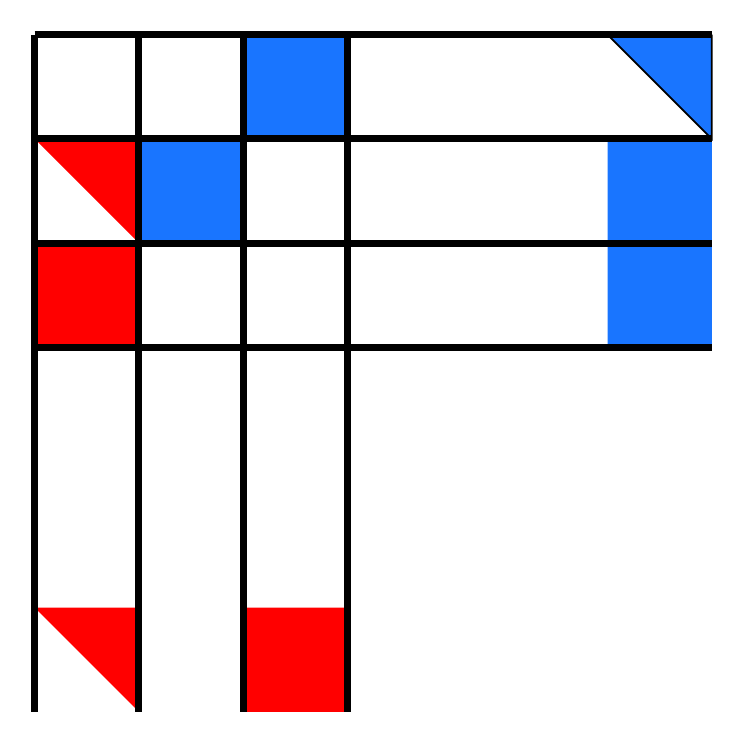}
  \captionof{figure}{}
  \label{fig7}
\end{minipage}%
\begin{minipage}{.45\textwidth}
  \centering
  \includegraphics[width=1\linewidth]{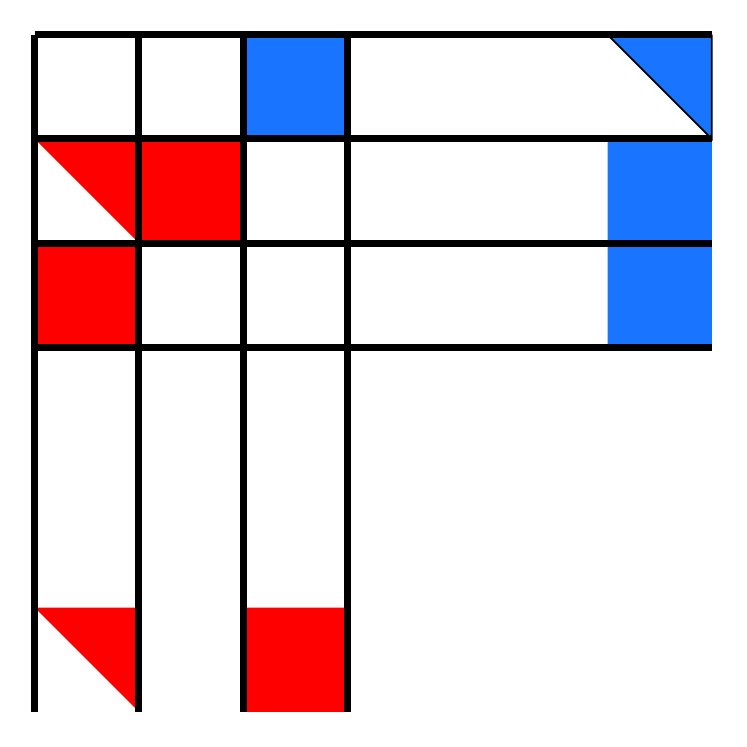}
  \captionof{figure}{}
  \label{fig8}
\end{minipage}
\captionsetup{labelformat=empty}
\caption{(Color coded figures) See the body of the article for the description.}
\end{figure}

\circled{H}
We now assume that $m_-(\calG)\geq 4$. Recall \eqref{f14.4}.
We will analyze $H_{m(\calH)-1}\cap B$. 
We apply transformations $\tnwc((k,m(\calH)-1), (m_-(\calG)-1, j))$ repeatedly for all $k < m_-(\calG)-1$ and $j< m_+(\calH)$.
The resulting $\bS$ has the  property 
\begin{align}
\P\left(A^c \cap \bigcup _{j < m_+(\calH)} C_{m_-(\calG)-1, j} \right) &= 0 \label{f9.10}
\quad \text{  or  }\\
\label{f9.9}
\P\left(A^c \cap \bigcup _{k < m_-(\calG)-1} C_{k, m(\calH)-1} \right) &= 0.
\end{align}

We apply the transformation $\calT_3(m_-(\calG)-1, m(\calH)-1)$. The resulting $\bS$ satisfies \eqref{f6.14} or \eqref{f6.15} with $(k,j) = (m_-(\calG)-1, m(\calH)-1)$.

\circled{H1}
If $\bS$ satisfies \eqref{f6.14} then we jump to \circled{A}. 
Then an application of $\calT_9$ will result in the  decrease of $m(\calG)$ or $m(\calH)$, in view of \eqref{f3.12} and \eqref{o4.3}.

\circled{H2}
Assume that \eqref{f6.15} holds, i.e., either $\P(C_{m_-(\calG)-1,m(\calH)-1}\cap A) =0$ or \break $\P(C_{m_-(\calG)-1,m(\calH)-1}\cap A^c) =0$. 

\circled{H3}
Suppose that 
\begin{align}\label{f16.1}
\P(C_{m_-(\calG)-1,m(\calH)-1}\cap A^c) =0.
\end{align}

\circled{H5}
Assume that \eqref{f9.10} holds. 
The transformations which generated \eqref{f9.9}-\eqref{f9.10} did not affect the conditions
\eqref{f9.2}-\eqref{f9.5}.
This, \eqref{f14.1}-\eqref{f14.2} and \eqref{f9.10} imply that $\P\left( \bigcup _{j < m(\calH)-1} C_{m_-(\calG)-1, j} \right) = 0$.
We have  $\P(C_{m_-(\calG)-1,m(\calH)}\cap A^c) =0$ because
of \eqref{f14.4}. These observations and \eqref{f16.1} imply that
 $x_{m_-(\calG)-1} = 1$. But this contradicts the definition of $m_-(\calG)$. Hence, \eqref{f9.10} cannot be true. \circled{X}
 
\circled{H6}
Next suppose that \eqref{f9.9} is true and recall 
\eqref{f14.1}-\eqref{f9.5}. We conclude that
 $\P(A^c \cap H_{m(\calH)-1})= 0$. Hence, $y_{m(\calH)-1} =1$. We have shown earlier that $y_{m(\calH)}=1$ so $y_{m(\calH)-1}=y_{m(\calH)}$. We jump to  \circled{A}. 
Then an application of $\calT_9$ will result in the  decrease of $m(\calG)$ or $m(\calH)$, in view of \eqref{f3.12} and \eqref{o4.3}.
 
\circled{H4}
Suppose that $\P(C_{m_-(\calG)-1,m(\calH)-1}\cap A) =0$. Then, in view of \eqref{f14.4},
\begin{align}\label{f11.1}
\P(C_{m_-(\calG),m(\calH)}\cap A^c) =0 \quad \text{  and  }\quad
 \P(C_{m_-(\calG)-1,m(\calH)-1}\cap A) =0.
\end{align}
Recall that we are assuming here that $m_-(\calG) \geq 4$.
The argument given between \circled{H} and  \circled{H6} can be applied  with the roles of $\calG$ and $\calH$, and those of $A$ and $A^c$, interchanged. Just like in the case of the original argument, some branches will end with \circled{X} and some will lead to \circled{A}. The only branch that will not end this way will generate the following analogue of \eqref{f11.1},
\begin{align}\label{f11.2}
\P(C_{1,m_+(\calH)}\cap A) =0 \quad \text{  and  }\quad
 \P(C_{2,m(\calH)+1}\cap A^c) =0.
\end{align}
The argument which lead to \eqref{f11.1} did not affect the cells in \eqref{f11.2} so, by analogy, the argument that can be used to prove \eqref{f11.2} does not affect the cells in \eqref{f11.1}. Hence,  both \eqref{f11.1} and \eqref{f11.2} hold.

For all $(k,j) \in \calD_-$ we have $\P(C_{k,j}) >0$ and either $\P(C_{k,j}\cap A^c) =0$ or $\P(C_{k,j}\cap A) =0$. It follows from \eqref{f11.1} and \eqref{f11.2} that there exist $(k,j) \in \calD_-$ with $\P(C_{k,j}\cap A^c) =0$ and $\P(C_{k+1,j+1}\cap A) =0$. Fix $(k,j)$ with these properties. 

We apply transformations $\tnwc((k+1,j+1), (k, j_1))$ repeatedly, for all $((k+1,j+1), (k, j_1))$ such that  $j_1\ne  j, j+1$. 
The resulting $\bS$ has the  property 
\begin{align}\label{f11.3}
\P\left(A^c \cap  C_{k+1, j+1} \right) &= 0,
\quad \text{  or  }\\
\P\left(A^c \cap \bigcup _{j_1 \ne j,j+1} C_{k, j_1} \right) &= 0.\label{f11.4}
\end{align}

 \circled{H7}
If \eqref{f11.3} holds then $\P(C_{k+1,j+1}) =0$ because we have assumed that $\P(C_{k+1,j+1}\cap A) =0$. 
We now apply the transformation defined in Lemma \ref{s29.3} (or one of its variants described at the beginning of the proof of Lemma \ref{o3.7}). This transformation decreases $m(\calG)$ or $m(\calH)$.
Then we jump to \circled{A}. 

 \circled{H8}
If \eqref{f11.4} holds then we combine it with  the assumption that $\P(C_{k,j}\cap A^c) =0$ to obtain
\begin{align}\label{f11.5}
\P(G_k \cap A^c) = \P(C_{k,j+1} \cap A^c).
\end{align}

Reversing the roles of $\calG$ and $\calH$, $(k,j)$ and $(k+1,j+1)$, and those of $A$ and $A^c$, we obtain the following formula analogous to \eqref{f11.5},
\begin{align}\label{f11.6}
\P(H_{j+1} \cap A) = \P(C_{k,j+1} \cap A).
\end{align}

It follows from \eqref{f11.5}, \eqref{f11.6}, $k\leq m_-(\calG)$ and $j+1\geq m_+(\calH)$ that
\begin{align*}
\P(A^c\mid C_{k,j+1})
&=\frac{\P(A^c \cap C_{k,j+1})}{\P(A^c \cap C_{k,j+1}) + \P(A \cap C_{k,j+1})}\\
&\geq
\frac{\P(A^c \cap C_{k,j+1})}{\P(A^c \cap C_{k,j+1}) + \P(A \cap C_{k,j+1})
+ \P(G_k \setminus C_{k,j+1})}\\
& = 
\frac{\P(A^c \cap G_k)}{\P(A^c \cap C_{k,j+1}) + \P(A \cap C_{k,j+1})
+ \P(G_k \setminus C_{k,j+1})}\\
&= 
\frac{\P(A^c \cap G_k)}{\P(G_k )}= 1 - x_k \geq 1- \delta,
\end{align*}
and
\begin{align*}
\P(A\mid C_{k,j+1})
&=\frac{\P(A \cap C_{k,j+1})}{\P(A^c \cap C_{k,j+1}) + \P(A \cap C_{k,j+1})}\\
&\geq
\frac{\P(A \cap C_{k,j+1})}{\P(A^c \cap C_{k,j+1}) + \P(A \cap C_{k,j+1})
+ \P(H_{j+1} \setminus C_{k,j+1})}\\
& = 
\frac{\P(A \cap H_{j+1})}{\P(A^c \cap C_{k,j+1}) + \P(A \cap C_{k,j+1})
+ \P(H_{j+1} \setminus C_{k,j+1})}\\
&= 
\frac{\P(A \cap H_{j+1})}{\P(H_{j+1} )}= y_{j+1}\geq 1-\delta.
\end{align*}
The two inequalities contradict each other since $\delta < 1/2$.
 \circled{X}

\circled{B}
We jump here from two points in the proof. We can jump here from \circled{A} in the case when $m_-(\calG) \leq 1$. We can also jump here from \circled{I}; in this case we have  $m_-(\calG) =2$ and  $\P(C_{1,m(\calH)})=0$.
All branches of the proof corresponding to $m_-(\calG) >2$ (sub-branches of \circled{J} and \circled{H}) ended with  \circled{X} or returned to \circled{A}. 

We have pointed out within the proof that every jump to \circled{A} was
associated with the decrease of $m(\calG)$ or $m(\calH)$, in most cases because of the application of $\calT_9$.
Hence, after a finite number of jumps to \circled{A}, $\bS$ must have been transformed so that $m_-(\calG) \leq 2$.

We can now exchange the roles of $\calG$ and $\calH$ and apply the same argument to cells $C_{k,j}$ with $k\geq m_+(\calG)$ and $j \leq m_-(\calH)$.
In this way, we will eliminate the case $m_-(\calH) >2$ and will be left with the case $m_-(\calH) \leq 2$. Moreover, if $m_-(\calH) =2$, we will have  $\P(C_{m(\calG),1})=0$.
Some transformations in the new part of the proof are of the type $\tnw
$ or $\tnwc$; these transformations do not change any $x_k$'s, $y_j$'s, $p_k$'s and $q_j$'s. Transformations of the type $\calT_3(\,\cdot\,,\,\cdot\,)$ in the new part of the proof do not affect cells  $C_{k,j}$ with $k\leq m_-(\calG)$ and $j \geq m_+(\calH)$.
Hence, we have constructed $\bS$ satisfying 
$m_-(\calG) \leq 1$, or  $m_-(\calG) =2$ and   $\P(C_{1,m(\calH)})=0$, and also satisfying $m_-(\calH) \leq 1$, or  $m_-(\calH) =2$ and   $\P(C_{m(\calG),1})=0$.

Part (i) of the lemma is satisfied because after the initial application of Lemma \ref{o1.2}, at the very beginning of the proof, all other transformations satisfied \eqref{f3.1}.
\end{proof}

\begin{proof}[Proof of Theorem \ref{f9.1}]
Recall that  $\delta \in (0, 1/2)$ is fixed.

For the proof of the lower bound in \eqref{s7.2}, see Proposition \ref{s28.4}.

We turn our attention to the upper bound.
In view of Lemma \ref{s28.5}, we can assume that $\calG$ and $\calH$ are finitely generated.

Consider any $\bS$. If $\P(B) =0$ then we are done. Assume that $\P(B)>0$ and consider any $\eps>0$. Let $\bS'$ be the result of the transformation defined in Lemma \ref{o3.10}.
Then $\P(B') \geq \P(B) -\eps$. 

Parts (ii) and (iii) of Lemma \ref{o3.10} imply that
for all $1\leq k \leq m'(\calG')$ there is at most one $j$ such that $C'_{k,j}\subset B'$ and $\P(C'_{k,j})>0$, and, similarly, for all $1 \leq j \leq m'(\calH')$ there is at most one $k$ such that $C'_{k,j}\subset B'$ and $\P(C'_{k,j})>0$.
Let $\calB'$ be the family of all $1\leq k \leq m'(\calG')$ such that there is a (unique) $j(k)$ such that $C'_{k,j(k)}\subset B'$. Note that $j(k_1) \ne j(k_2)$ if $k_1 \ne k_2$.
We  apply Lemma \ref{f11.7}
as follows.
\begin{align*}
\P(B)-\eps &\leq \P(B') = \sum_{k\in\calB'}\P(C'_{k,j(k)}) 
= \sum_{k\in\calB'} \P(G'_{k}\cap H'_{j(k)})
\leq \sum_{k\in\calB'} \frac \delta{1+\delta} (\P(G'_{k}) + \P(H'_{j(k)}))\\
&=\frac \delta{1+\delta}\left (
 \sum_{k\in\calB'}  \P(G'_{k})+ \sum_{k\in\calB'}  \P(H'_{j(k)})\right)
\leq \frac {2\delta}{1+\delta} .
\end{align*}
Since $\eps>0$ is arbitrarily small, $\P(B) \leq 2\delta/(1+\delta)$.
\end{proof}

\section{Acknowledgments}
We are grateful to Jim Pitman for very useful advice.

\bibliographystyle{plain}
\bibliography{pred}

\end{document}